\documentclass[11pt]{article}
\usepackage{makeidx}
\usepackage[centertags]{amsmath}
\usepackage{amsthm}
\usepackage{newlfont}
\usepackage{amsmath}
\usepackage{amsfonts,amsmath,latexsym}
\usepackage{mathrsfs}
\usepackage{amssymb}
\usepackage{amsbsy}
\usepackage{indentfirst}
\usepackage{mathrsfs}
\usepackage{graphicx}
\usepackage{epstopdf}
\usepackage{subfigure}
\usepackage{caption}
\usepackage{grffile}
\usepackage{float}
\usepackage{amssymb, amsmath}
\usepackage{xcolor}
\usepackage{hyperref}
\usepackage{CJK, CJKnumb, CJKulem, verbatim}
\usepackage{appendix}

\hypersetup{
    colorlinks=true,
    linkcolor=blue, 
    filecolor=magenta,      
    urlcolor=cyan,
    pdftitle={Your Title},
    pdfauthor={Your Name},
    pdfsubject={Subject},
    pdfkeywords={Keywords},
    bookmarks=true,
    bookmarksopen=true,
    bookmarksnumbered=true,
    pdfpagemode=UseOutlines,
    pdfstartview=Fit,
}

\titlepage
\newlength{\defbaselineskip}
\newcommand{\be}{\begin{eqnarray}}
\newcommand{\ee}{\end{eqnarray}}
\newcommand{\bestar}{\begin{eqnarray*}}
\newcommand{\eestar}{\end{eqnarray*}}

\newcommand{\ignore}[1]{}
{} \theoremstyle{plain}
\newtheorem{thm}{Theorem}[section]

\newtheorem{lemma}{Lemma}[section]
\newtheorem{prop}{Proposition}[section]

\theoremstyle{definition}

\newtheorem{rem}{Remark}[section]
\numberwithin{equation}{section}

\allowdisplaybreaks
\graphicspath{{./graphics/}}

\begin{document}

\title{Berry-Esseen bounds for step-reinforced random walks}
\author{Zhishui Hu
\thanks{E-mail: huzs@ustc.edu.cn}\\
Department of Statistics and Finance,  School of Management\\
 University of Science and Technology of
China\\
  Hefei, Anhui 230026, China}
\date{}
\maketitle

\begin{abstract}
We study both the positively and negatively step-reinforced random walks with  parameter $p$. For a step distribution $\mu$ with finite second moment, 
the positively  step-reinforced random walk with $p\in [1/2,1)$ 
and the negatively step-reinforced random walk with $p\in (0,1)$ converge to a normal distribution under suitable normalization. In this work, we  obtain   the rates of convergence to normality  for both cases under the assumption that $\mu$ has a
finite third moment. In the proofs, we establish a Berry-Esseen bound for general functionals of independent random variables, utilize the randomly weighted sum  representations of step-reinforced random walks,  and apply special comparison
arguments to quantify the Kolmogorov distance between a mixed normal distribution and its corresponding normal distribution.

 \vskip 0.2cm

\noindent{\it MSC2020 subject classifications:}
   60G50;    
   60F05;     
   05C05   
\end{abstract}

\section{Introduction and main results} \label{sectintro}

 Step-reinforced random walks, as a class of stochastic processes with memory,  have  garnered considerable attention  in recent years. Among these,
the elephant random walk (ERW) serves as a fundamental example.
The ERW is a one-dimensional discrete-time random walk on $\mathbb{Z}$ 
that retains complete memory of its entire history.
 First introduced by Sch\"{u}tz and Trimple \cite{SCH2004}, the ERW 
is characterized by a fixed parameter $q\in [0,1]$, referred to as  the memory parameter. The walk starts at position $0$ at time $n=0$,
with its initial step  determined by a symmetric Rademacher random variable taking 
values $+1$ or $-1$ with equal probability.
At each subsequent time $n\ge 2$, the ERW randomly chooses one of its previous steps.
 It then repeats that step with probability $q$ or takes an opposite step with probability $1-q$.
The asymptotic behaviour of the ERW has been extensively studied, see, for instance,
\cite{BB2016, BE2017, Col2017a, Col2017b, Fan2021, Fan2024, KT2019, KU2016, Qin2025}.

For $q\ge 1/2$, K\"{u}rsten \cite{KU2016} proposed an alternative characterization  of the ERW dynamic by introducing a new parameter $p=2(1-q)\in [0,1]$.
The initial step remains a symmetric Rademacher random variable.
However, at each step $n\ge 2$, the ERW either repeats one of its previous steps, chosen uniformly at random, with probability $1-p$, or
it takes a new independent symmetric Rademacher random variable with probability $p$. By construction,  each step  of the ERW follows the Rademacher distribution.
This framework  generalizes naturally  to arbitrary distributions on $\mathbb{R}$, denoted by $\mu$.
When  $\mu$ is an isotropic stable law,  the model is termed  the ``shark random swim" by Businger \cite{BU2018}.
More generally,  for any distribution $\mu$,  the model is defined as the positively step-reinforced random walk, which has recently been
investigated, for example, in \cite{B2021, BR2022, Bertoin2020, Bertoin2021A, Bertoin2021B, HZ2024, Qin2024}.

The  positively step-reinforced random walk is formally constructed as follows.
Let  $X_1, X_2,\cdots$ be a sequence of  independent and identically distributed (i.i.d.) random variables with distribution $\mu$.    Define
  \begin{equation}
  m_k= \mathbb{E}(X_1^k), ~~k\ge 1,~~~\mbox{and}~
  ~~ \sigma_0^2=\mbox{Var}(X_1)=m_2-m_1^2.   \label{meanvar}
  \end{equation}
 Let $\epsilon_1=1$, and let $\epsilon_2, \epsilon_3, ...$  be i.i.d. Bernoulli variables  with parameter $p \in [0,1]$.
 Let $(U_n)_{n \ge 2}$ be a sequence of independent random variables, where 
each $U_n$ is uniformly distributed on  $\{1, \cdots, n-1\}$. It is further assumed that $(X_n), (U_n)$, and $(\epsilon_n)$ are independent.
Define
\be
\text{i}(n):=\sum_{j=1}^{n} \epsilon_j \quad \quad \text{for} \quad n \ge 1. 
\label{in}
\ee
Set $\hat{X}_1=X_1$, and for $n\ge 2$,  recursively  define
  \be
   \hat{X}_n =
    \begin{cases}
      \hat{X}_{U_n}, & \mbox{if } \epsilon_n =0, \\
      X_{\text{i}(n)}, & \mbox{if } \epsilon_n =1.
    \end{cases}  \label{Xnhat}
\ee
The sequence of the partial sums
\be
\hat{S}_n=\sum_{i=1}^n \hat{X}_i,~~~~n\ge 1,     \label{hats}
\ee
 is referred to as a positively step-reinforced random walk or a noise reinforced random walk.
The  reinforcement algorithm (\ref{Xnhat}) was introduced by Simon \cite{Simon1955} to explain the appearance of a family of heavy-tailed distributions
in a wide range of empirical data.
When $\mu$ is the  symmetric Rademacher distribution, $(\hat{S}_n)_{n\ge 1}$ corresponds to the
ERW with memory parameter $q=1-p/2\in [1/2,1]$.
The ERW with  memory parameter $q$ in the remaining range $[0,1/2)$ can also be obtained as a special case of the
negatively step-reinforced random walk, introduced by Bertoin \cite{Bertoin2024}.

Set $\check{X}_1=X_1$, and for $n\ge 2$,   recursively define
  \bestar
   \check{X}_n =
    \begin{cases}
      -\check{X}_{U_n}, & \mbox{if } \epsilon_n =0, \\
      X_{\text{i}(n)}, & \mbox{if } \epsilon_n =1.
    \end{cases}
\eestar
Then the process
\be
\check{S}_n=\sum_{i=1}^n \check{X}_i,~~~~n\ge 1,    \label{checks}
\ee
 is referred to as a negatively step-reinforced random walk or a counterbalanced random walk.
The negatively step-reinforced random walk  has been studied  in 
\cite{ BR2022, Bertoin2024, HZ2024}.
When $\mu$ is the  symmetric Rademacher distribution,  $(\check{S}_n)_{n\ge 1}$ corresponds to  the
ERW with memory parameter $q=p/2\in [0, 1/2]$.

Note that if $p=1$,   both $(\hat{S}_n)_{n\ge 1}$ and  $(\check{S}_n)_{n\ge 1}$ reduce to standard random walks with i.i.d. steps. If $p=0$,  the positively step-reinforced random walk satisfies  $\hat{X}_n=X_1$ for all $n\ge 1$, while  the negatively step-reinforced random walk has steps $\check{X}_n$ equal to $X_1$ or $-X_1$. In this paper, we exclude these trivial cases 
 and always assume that $p\in (0,1)$.

Assume that $\mathbb{E}(X_1^2)<\infty$. When $p<1/2$, Theorem 1 in  \cite{Bertoin2021A}  shows that $n^{-p}(\hat{S}_n-m_1n)$
converges in $L^2(\mathbb{P})$ to some non-degenerate random variable $L$.
When $p\ge 1/2$, by Theorem 2 in  \cite{Bertoin2021A}  and Theorems 1.2 and 1.4 in \cite{BR2022}, we have
\be
\frac{\hat{S}_n-nm_1}{\sigma_0\sqrt{a_n}}\stackrel{d}{\rightarrow} N(0,1), \label{CLThatS}
\ee
where
\be
a_n=\left\{
\begin{array}{ll}
n/(2p-1), & p>1/2,\\
n\log n, & p=1/2.
\end{array}
\right.    \label{an}
\ee
Regarding $\check{S}_n$,
Bertoin \cite{Bertoin2024} established the central limit theorem,  which shows that
\bestar
\frac{\check{S}_n-\check{b}\, n}{\check{\sigma}\sqrt{n}}\stackrel{d}{\longrightarrow} N(0,1)
\eestar
with
\be
~\check{b}:=\frac{pm_1}{2-p}, ~~~~~~~~\check{\sigma}^2=\frac{m_2-\check{b}^2}{3-2p},   \label{checkb}
\ee
where $m_1$ and $m_2$ are defined in (\ref{meanvar}).
Bertenghi and Rosales-Ortiz \cite{BR2022} proved the functional central limit theorems for $\hat{S}_n$ and $\check{S}_n$ through the martingale method.

In this paper,  we aim to establish the Berry-Esseen bounds for $\hat{S}_n$ and $\check{S}_n$.
Let $\mathbf{Z}$ be a standard normal random variable.
For any random variables $X$ and $Y$, we denote the Kolmogorov distance between $X$ and $Y$ by
\bestar
d_K(X, Y)=\sup_{x\in \mathbb{R}} |\mathbb{P}(X\le x)-\mathbb{P}(Y\le x)|.
\eestar

\begin{thm} \label{th3} 
Let $m_1$ and $\sigma_0^2$ be defined in  (\ref{meanvar}).
Assume that $\sigma_0^2>0,  ~\mathbb{E}(|X_1|^3)<\infty$ and $p\in [1/2,1)$. 
Then
\be
d_K\Big(\frac{\hat{S}_n-m_1 n}{\sigma_0 \sqrt{b_n}}, \mathbf{Z}\Big)\le C\delta_{1,n},    \label{BEhatSn}
\ee
where 
\be
b_n=\left\{
\begin{array}{ll}
\frac{n}{2p-1}- \frac{n^{2-2p}}{(2p-1)\Gamma(2-2p)}, &  p>1/2, \\
n\log n +\gamma n, & p=1/2,
\end{array}
\right.  \label{bn2}
\ee
$\Gamma(s)=\int_0^{\infty} x^{s-1} e^{-x} dx~(s>0)$ is the Gamma function, $\gamma=\lim_{n\rightarrow \infty} (\sum_{k=1}^n k^{-1}-\ln n)$ is Euler's constant and
\be
\delta_{1,n}=\left\{
\begin{array}{ll}
n^{-1/2}, & p>2/3,\\
n^{-1/2}\log n, & p=2/3,\\
n^{3/2-3p}, & 1/2<p<2/3,\\
 (\log n)^{-3/2}, & p=1/2.
 \end{array}
\right.   \label{delta1n}
\ee
\end{thm}

\begin{rem} \label{rembn}
 Let $a_n$ be defined in (\ref{an}).  Asymptotically,  we have $a_n\sim b_n$ as $n\rightarrow \infty$. Consequently,  (\ref{CLThatS}) remains valid when $a_n$ is replaced by $b_n$.
In order to obtain a better convergence rate,  we use $b_n$ instead of $a_n$ in (\ref{BEhatSn}). 
Notably,   for  $p\in [1/2, 1)$, the sequence $b_n$ is positive. It is obvious for $p=1/2$. And for  $p\in (1/2, 1)$, we have  $\Gamma(2-2p)>\Gamma(1)=1$ since $\Gamma(s)$ is strictly decreasing on $(0,1)$, and hence $b_n>0$. 
\end{rem}

\begin{thm} \label{th4} 
Let $\check{b}$ and $\check{\sigma}^2$ be definded in (\ref{checkb}).
Assume that $\check{\sigma}^2>0, ~\mathbb{E}(|X_1|^3)<\infty$ and $p\in (0,1)$.  Then
\be
d_K\Big(\frac{\check{S}_n-\check{b} n}{\check{\sigma}\sqrt{n}}, ~\mathbf{Z}\Big)\le C\delta_{2,n},
\label{checkSnBE}
\ee
where
\be 
\delta_{2,n}=\left\{
\begin{array}{ll}
 n^{-1/2}, & p>1/3;\\
 n^{-1/2}\log n, & p=1/3;\\
 n^{-3p/2}, & 0< p< 1/3.
 \end{array}  
\right. \label{delta2n}
\ee 
\end{thm}

\begin{rem}\label{rem1.2}
Recall that when $\mu$ is the  symmetric Rademacher distribution, $\hat{S}_n$ and $\check{S}_n$ correspond to the
ERW with memory parameters $q=1-p/2\in [1/2,1]$ and $q=p/2\in [0,1/2]$, respectively.   Berry-Esseen bounds for the ERW have been established in \cite{Fan2021, Fan2021a, Fan2024, HOT2023, Qin2025}. The best existing  convergence  rate for the ERW was otained in Theorem 3 of \cite{Fan2021}, which derived the bounds
\be
&&d_K\Big(\frac{\hat{a}_n\hat{S}_n-(1-p)}{\sqrt{\hat{v}_n}}, \mathbf{Z}\Big)\le  \left\{
\begin{array}{ll}
Cn^{-1/2}, & 3/4\le p<1,\\
C\hat{v}_n^{-1}, &  1/2\le p<3/4,
\end{array}\right.   \label{ERW1}
\ee
and 
\be 
&& d_K\Big(\frac{\check{a}_n\check{S}_n+(1-p)}{\sqrt{\check{v}_n}}, \mathbf{Z}\Big)\le Cn^{-1/2},\qquad 0<p<1,  \label{ERW2}
\ee 
where $\hat{a}_1=\check{a}_1=1$, and for $n\ge 2$,
\bestar
&&\hat{a}_n=\frac{\Gamma(n)\Gamma(2-p)}{\Gamma (n+1-p)}\sim n^{p-1},~~~~
\hat{v}_n=\sum_{k=1}^n \hat a_k^2\sim \left\{
\begin{array}{ll}
Cn^{2p-1}, & p>1/2,\\
C\log n, & p=1/2.
\end{array}\right.\\
&&\check{a}_n=\frac{\Gamma(n)\Gamma(p)}{\Gamma (n+p-1)}\sim n^{1-p},~~~~
\check{v}_n=\sum_{k=1}^n \check a_k^2\sim 
Cn^{3-2p},~~0<p<1.
\eestar
Comparing the above results with  Theorems \ref{th3} and \ref{th4}
in this special case,  the convergence rate in Theorem \ref{th3} is better than  \eqref{ERW1} 
for $1/2\le p< 3/4$,  whereas the  rate in Theorem \ref{th4} is weaker than \eqref{ERW2} for $p\le 1/3$. 
Note that when $p=0$, it follows immediately from Section 2 of \cite{Bertoin2024}  that $\check{S}_n/\sqrt{n}$ 
converges to $N(0,1/3)$ at a rate of $O(n^{-1/2})$ under  the symmetric Rademacher distribution.  In contrast,  for a general distribution $\mu$, $\check{S}_n$ cannot be normalized to converge to a normal distribution
because,  at $p=0$, $\check{X}_n$ equals either $X_1$ or $-X_1$. Therefore, in some sense, it is reasonable in Theorem \ref{th4} that for small $p$, $\check{S}_n$ fails to achieve a convergece  rate of order  $O(n^{-1/2})$.
\end{rem}

\begin{rem}
The estimates of the Kolmogorov distance in Theorems \ref{th3} and \ref{th4}
involve the quantities $Z_{3}(n)$ and $Z_{3/2}(n)$, which are defined in 
Lemma  \ref{lemma38}.  As shown in  Lemma  \ref{lemma38},
the moments of these quantities exhibit a phase transition at $p=2/3$
and $p=1/3$, respectively. This explains why these two values appear 
as thresholds in  Theorems \ref{th3} and \ref{th4}.  The same lemma 
also clarifies why   the upper bounds in these theorems contain a logarithmic factor at these thresholds.  Separately,  in Theorem \ref{th3},
 the factor $(\log n)^{-3/2}$ at $p=1/2$ originates from the order of magnitude of the variance, i.e.,
$\mbox{Var} (\hat{S}_n) \asymp n\log n$.
\end{rem}

Our main results, Theorems \ref{th3} and \ref{th4}, will be proved  based on  the fact that both step-reinforced random walks, $\hat{S}_n$ and $\check{S}_n$, can be expressed as randomly weighted sums
(see (\ref{hatsnsum}) and (\ref{checksnsum})).
It is noted that, conditioned on an appropriate $\sigma$-filed,
these randomly weighted sums can be regarded as  sums of independent random variables. 
By applying the classical Berry-Esseen theorem, 
we derive an upper bound for the Kolmogorov distance between the distribution 
 of $(\hat{S}_n-m_1 n)/(\sigma_0 \sqrt{b_n})$ in (\ref{BEhatSn}) (or $(\check{S}_n-\check{b} n)/(\check{\sigma}\sqrt{n})$ in (\ref{checkSnBE}))  and a mixed normal distribution.
 Theorems \ref{th3} and \ref{th4} are then obtained  through  special comparison arguments that quantify  the Kolmogorov distance between a mixed normal distribution and its corresponding normal distribution.

The remainder of this work is organized as follows.
Section \ref{sect2} presents preliminary tools and fundamental results.
In Section \ref{subsection2.1}, 
we  express $\hat{S}_n$ and $\check{S}_n$ as randomly weighted sums.
Section \ref{subsection2.2} constructs the associated   random recursive tree  together with a percolation process on it,  and provides some preliminary lemmas.  In Section \ref{subsection2.3}, we present a Berry-Esseen theorem (Proposition \ref{lemmanu1nBE}) that is key to the proof of Theorem \ref{th4}. To establish this result, we first introduce a Berry–Esseen theorem for general functionals of independent random variables and then apply it to derive
two related Berry-Esseen type results.

 Section \ref{secta2}  is devoted to the proofs of Theorems \ref{th3} and \ref{th4}.
Section \ref{secta4} provides proofs of the results in  Section  \ref{subsection2.3}. 
The proofs of the lemmas in Section \ref{subsection2.2} and of an inequality from Section \ref{secta4}, both of which involve detailed moment estimates for a variety of quantities, are given in  \ref{secta3} and \ref{sectappendix2}.

Throughout this paper,  $C$ is a positive constant not depending on $n$ that may take a different value in each appearance.
We use $O(\cdot)$ to denote a quantity that is bounded in absolute value by the quantity in the
parentheses multiplied by a constant not depending on $n$.
To simplify notation, let $x\vee y$ and $x\wedge y$ be the  maximum and minimum of $x$ and $y$, respectively.
For two sequences of positive numbers $(c_n)$ and $(d_n)$, we write $c_n \asymp d_n$
if and only if $0<\liminf_{n\rightarrow\infty} c_n/d_n\le \limsup_{n\rightarrow\infty} c_n/d_n<\infty$.

\section{Technical tools and preliminary results} \label{sect2}

In this section, we introduce key tools
and  preliminary results that  will be used in the proofs of  Theorems \ref{th3} and \ref{th4}.

\subsection{Randomly weighted sums} \label{subsection2.1}

To analyze $\hat{S}_n$ and $\check{S}_n$, we first use the approach  of Bertoin \cite{Bertoin2024} to represent them as randomly weighted sums of the underlying i.i.d. sequence $(X_n)$.

For every $n,j\in \mathbb{N}$, we write
\be
N_j(n):=\#\{l\le n: \hat{X}_l=X_j\}    \label{defNjn}
\ee
for the number of occurrences of the variable $X_j$ in the sequence $\{\hat{X}_l: 1\le l\le n\}$, and
\be
\nu_k(n):=\#\{1\le j\le i(n): N_j(n)=k\},~~~~k\in \mathbb{N} \label{nuk}
\ee
for the number of such variables that have occurred exactly $k$ times.
It follows from the definition of $\hat{S}_n$ that $\{N_j(n), ~1\le j\le n, ~n=1,2,\cdots\}$
is independent of $\{X_j, j=1, 2,\cdots\}$, and
\be
\hat{S}_n=\sum_{j=1}^n N_j(n) X_j.    \label{hatsnsum}
\ee

In our study of $\check{S}_n$, 
we adopt the notation defined in  \cite{Bertoin2024}.
For any $n\ge 1$ and $1\le j\le i(n)$, let $l_1<l_2<\cdots l_k$ be the increasing sequence of steps at which $X_j$
appears in $\{\hat{X}_l: 1\le l\le n\}$, where $k=N_j(n)\ge 1$. We define $T_j(n)$ as a rooted tree on $\{1,2,\cdots, k\}$ with root $1$ such that
for every $1\le a<b\le k$, $(a, b)$ is an edge of $T_j(n)$ if and only if $U_{l_b}=l_a$. 
By convention,  assume that $T_j(n)$ is the empty graph if $i(n)<j\le n$.
For any rooted tree $T$, let $\Delta(T)$
denote the difference obtained by subtracting  the number of vertices at odd distances from the root from the number  at even distances. 
Then we have
\be
\check{S}_n=\sum_{j=1}^{n} \Delta(T_j(n)) X_j,    \label{checksnsum}
\ee 
where $\{\Delta(T_j(n)), ~1\le j\le n, ~n=1,2,\cdots\}$
is independent of $\{X_j, j=1, 2,\cdots\}$.

\subsection{Random recursive trees} \label{subsection2.2}
 K\"{u}rsten \cite{KU2016} revealed a connection between the ERW and Bernoulli bond percolation on random recursive trees. The processes $\hat{S}_n$ and $\check{S}_n$ are closely related to
random recursive trees. Specifically,
let $(U_i)_{i\ge 2}$ be defined as in Section \ref{sectintro} and
consider $\mathbb{T}_n$, the random graph with  vertex set $\{1,2, \cdots, n\}$ and
edge set $\{(U_i, i): i=2,\cdots, n\}$.  $\mathbb{T}_n$ is thereby a random recursive tree of size $n$.
Random recursive tree have been extensively studied for their various theoretical properties and applications. For more details, we refer to \cite{M1995}  and references therein.

Now, let $(\varepsilon_i)_{i\ge 2}$ be defined as in Section \ref{sectintro}.
We can construct Bernoulli bond percolation on $\mathbb{T}_n$ with survival parameter $1-p\in [0, 1]$
as follows: for  $2\le i\le n$, the edge $(U_i, i)$ in $\mathbb{T}_n$ is open if $\varepsilon_i=0$ and closed if $\varepsilon_i=1$.
Moreover, the quantity $\nu_k(n)$ defined in (\ref{nuk}) is the number of percolation clusters of size $k$.

The proofs of the main results require several moment conclusions about  $\nu_k(n)$ and the vertex degrees  in the random recursive tree, as provided in Lemmas \ref{lemma1}$-$\ref{lemma38}. We first present some necessary definitions.
For all $n,i, j\in \mathbb{N}$ with $1\le i, j\le n$, define
\be
 I_{i,j}=I(U_j=i) \label{Iij},   \label{Iij}
\ee
and let $D_{n,i}$ denote the  degree of  vertex $i$ in the random recursive tree $\mathbb{T}_n$, given by
\be
D_{n,i}=I(i\ne 1)+\sum_{j=i+1}^n I_{i,j}.  \label{Dni}
\ee
The relevant lemmas are stated below.

\begin{lemma} \label{lemma1} For any $l\in \mathbb{N}$, we have
\bestar
\sum_{i=1}^n  \mathbb{E}(D_{n,i}^l)=O(n).
\eestar
\end{lemma}

\begin{lemma} \label{lemmaadd2} We have
\bestar
\mathbb{E}(\nu_1(n))=\frac{np}{2-p}+O(1), ~~~~~~
\mathbb{E}(\nu_2(n))=\frac{np(1-p)}{(2-p)(3-2p)}+O(1),
\eestar
and $\mbox{Var}(\nu_i(n))\le Cn$ for $i=1,2.$
\end{lemma}

\vskip 0.3cm

\begin{lemma} \label{lemma38}  Define
$
Z_{l}(n)=\sum_{k=1}^n k^l \nu_k(n)
$
for $l\ge 0$.
For any $0<p<1$, we have
\be
\mathbb{E}(Z_{l}(n))  \asymp b_{l}(n),   \label{Zln}
\ee
where
\be
b_{l}(n)=\left\{
\begin{array}{ll}
n^{l(1-p)}, & l(1-p)>1;\\
 n\log n,  & l(1-p)=1;\\
 n, & l(1-p)<1.
\end{array}
\right.   \label{bln}
\ee
Moreover,
\be
\mathbb{E}(Z_{2}(n))=b_n (1+O(n^{-1}))
~~~~and~~~~
\mbox{Var}(Z_{2}(n))\le Cb_4(n),  \label{EZ2na}
\ee
 where  $b_4(n)$ is defined in (\ref{bln}) with $l=4$ and
\bestar
b_n=\left\{
\begin{array}{ll}
\frac{n}{2p-1}- \frac{n^{2-2p}}{(2p-1)\Gamma(2-2p)}, &  p\ne 1/2, \\
n\log n +\gamma n, & p=1/2.
\end{array}
\right.  \label{bn3}
\eestar
\end{lemma}
\vskip 0.3cm

The proofs of Lemma \ref{lemmaadd2}  and  Proposition \ref{lemmanu1nBE} below, which involve the moments  of $\nu_1(n)$ and 
 the convergence rate of $\nu_1(n)$, respectively, are based on a conditioning argument for  $\nu_1(n)$ with respect to $\mathbb{T}_n$, or more precisely, with respect to the sigma-algebra $\sigma(U_2, \dots, U_n)$.  This approach requires estimates for  the mean and variance 
of both the conditional expectation $\mu(\mathbb{T}_n):=\mathbb{E}(\nu_1(n)|\mathbb{T}_n)$ and the conditional variance
$\sigma^2(\mathbb{T}_n):=\mbox{Var}(\nu_1(n)|\mathbb{T}_n)$, as provided in Lemmas \ref{lemma43}  and \ref{lemmaadd1} below.

As $\nu_1(n)$ denotes the number of isolated vertices in the Bernoulli bond percolation on $\mathbb{T}_n$, we can express $\nu_1(n)$ as
$
\nu_1(n)=\sum_{i=1}^n J_i,
$
where $J_i=1$ if  $i$ is an isolated vertex and $J_i=0$ otherwise.
Observing that $\mbox{Cov}(J_i, J_j | \mathbb{T}_n)=0$ if $(i,j)\not\in \mathbb{T}_n$,  simple calculations show that 
\be
\mu(\mathbb{T}_n)&=&\mathbb{E}(\nu_1(n)|\mathbb{T}_n)=\sum_{i=1}^n \mathbb{E}(J_i|\mathbb{T}_n)=\sum_{i=1}^n p^{D_{n,i}},\label{munT}\\
\sigma^2(\mathbb{T}_n) &=& \mbox{Var}(\nu_1(n)|\mathbb{T}_n)
=\sum_{i=1}^n \mbox{Var}(J_i|\mathbb{T}_n)+\sum_{(i,j)\in \mathbb{T}_n} \mbox{Cov}(J_i,J_j|\mathbb{T}_n)\nonumber\\
&=&\sum_{i=1}^n (p^{D_{n,i}}-p^{2D_{n,i}})+\sum_{(i,j)\in \mathbb{T}_n}p^{D_{n,i}+D_{n,j}-1} (1-p)\nonumber\\
&=&\sum_{i=1}^n (p^{D_{n,i}}-p^{2D_{n,i}})+2(1-p)\sum_{1\le i<j\le n}I_{i,j}p^{D_{n,i}+D_{n,j}-1}. \label{sigmanT}
\ee

\begin{lemma}\label{lemma43} 
We have
\bestar
\mathbb{E}(\mu(\mathbb{T}_n))=\frac{np}{2-p}+O(1)~~ \mbox{and}~~
\mbox{Var}(\mu(\mathbb{T}_n))=\sigma_{1}^2 n+O(1),
\eestar
where
\be
 \sigma_{1}^2=\frac{2p^2(1-p)^4}{(2-p^2)(2-p)^2(3-2p)}. \label{sigma1}
\ee
\end{lemma}

\begin{lemma} \label{lemmaadd1}
We have 
\bestar
\mathbb{E}(\sigma^2(\mathbb{T}_n))=\sigma_{2}^2 n+O(1)~~ \mbox{and}~~\mbox{Var}(\sigma^2(\mathbb{T}_n))\le Cn,
\eestar
 where
\be
\sigma_{2}^2=\frac{2p(1-p)(3-p^3)}{(2-p)(2-p^2)(3-2p)}.  \label{sigma2}
\ee
\end{lemma}

 Detailed  proofs of the above lemmas  are provided in   \ref{secta3}.

\subsection{A Berry-Esseen theorem for functionals of independent random variables with applications} \label{subsection2.3}

Let  $\xi=(\xi_1, \cdots, \xi_n)$ be a vector of independent  random variables, and
let $\xi'=(\xi_1', \cdots, \xi_n')$ be an independent copy of $\xi$.
For  any $A\subseteq \{1,2,\cdots, n\}$, we define the random vector $\xi^A$ as
\be
\xi^{A}=(\xi_1^A, \xi_2^{A},\cdots, \xi_n^{A}),    \label{XA}
\ee
where
\bestar
\xi_i^A=\left\{
\begin{array}{cc}
	\xi_i, & \mbox{if}~ i\not\in A,\\
	\xi_i', & \mbox{if}~i\in A.
\end{array}
\right.
\eestar
Suppose that $f: \mathbb{R}^n\rightarrow \mathbb{R}$ is a measurable function such that $\mathbb{E}(f^2(\xi))<\infty$. 
Set
\bestar
\sigma^2=\mbox{Var}(f(\xi))~~~~\mbox{and} ~~~~
W=\frac{f(\xi)-\mathbb{E}(f(\xi))}{\sigma}.
\eestar

\begin{thm} \label{thBE1}
	Define $\{\Delta_i, 1\le i\le n\}$ and $\{\Delta_{ij}, 1\le i\ne j\le n\}$ by
	\bestar
	\Delta_i= \mathbb{E}(f(\xi)-f(\xi^{\{i\}})|\xi_1,\cdots, \xi_i, \xi_i'), ~~1\le i\le n
	\eestar
 and
	\bestar
	\Delta_{ij}=\mathbb{E}(f(\xi)-f(\xi^{\{i\}})-f(\xi^{\{j\}})+f(\xi^{\{i,j\}})|\xi_1,\cdots, \xi_{i\vee j}, \xi_i', \xi_j'),~~1\le i\ne j \le n.
	\eestar
	Then 
	\be
	d_K(W, \mathbf{Z})\le \frac{5}{\sigma^2} \Big(\sum_{j=1}^n\mathbb{E}\Big(\sum_{i\ne j} |\Delta_i||\Delta_{ij}|\Big)^2
	+\sum_{i=1}^n   \mathbb{E}(\Delta_i^4)\Big)^{1/2}.
	\label{corthBE1}
	\ee
\end{thm}

\begin{rem}
Shao and Zhang \cite{Shao2025} adopted the idea from \cite{Chatt2008} and used the  differential variables $f(\xi)-f(\xi^{\{i\}})$ and $f(\xi)-f(\xi^{\{i\}})-f(\xi^{\{j\}})+f(\xi^{\{i,j\}})$ to obtain a bound  for  $d_K(W, \mathbf{Z})$
  (see Corollary 2.5 therein). However, the related conclusion  in \cite{Shao2025} requires the assumption that $\xi_1, \cdots, \xi_n$
  are identically distributed, whereas in this paper, we need to handle the case where 
$\xi_1, \cdots, \xi_n$ are  not identically distributed.
\end{rem}

	In the remainder of this subsection, 
	 we first present two   Berry-Esseen type results (Propositions  \ref{proppercolation} and \ref{propa41}) as applications of Theorem \ref{thBE1}.

The first proposition establishes a Berry-Esseen bound
for the number of vertices with a specified  degree 
 in  Bernoulli bond percolation on  general finite graphs.
Let $G_n$ be
 a graph with vertex set $V_n=\{1,2,\cdots, n\}$ and edge set  $E_n=\{e_1, \cdots, e_{m}\}$. Denote by $d_{n,i}$  the  degree of vertex $i$ in $G_n$.
Consider Bernoulli bond percolation on $G_n$ with parameter $\widetilde{p} \in (0, 1)$, where each edge of $G_n$
is independently open with probability $\widetilde{p}$. 
Let $N_{n,d}$ denote the number of vertices with  degree $d\ge 0$ in  the percolated subgraph.

\begin{prop}  \label{proppercolation} Suppose  that $\mu_{n,d}=\mathbb{E}(N_{n,d})$ and $\sigma_{n,d}^2=\mbox{Var}(N_{n,d})>0$. Then we have
\bestar
d_K \Big(\frac{N_{n,d}-\mu_{n,d}}{\sigma_{n,d}}, ~\mathbf{Z}\Big) \le \frac{C}{\sigma_{n,d}^2} \Big(m\widetilde{p}\,(1-\widetilde{p}\,)+ \widetilde{p}\,^3(1-\widetilde{p}\,)^3 \,\sum_{i=1}^n d_{n,i}^3\Big)^{1/2},  \label{BEpercola}
\eestar
where $C>0$ is an absolute constant.
\end{prop}

\begin{rem}Let  $K_n$ denote the complete graph   with vertex set $V_n=\{1,2,\cdots, n\}$.   Bernoulli bond percolation on $K_n$ with parameter $\widetilde{p}$
generates the well-known  Erd\H{o}s-R\'{e}nyi graph $G(n,\widetilde{p})$.
Goldstein \cite{Gold2013} and Krokowski et al. \cite{K2017} derived  related Berry-Esseen bounds.

We investigate the special case where $\widetilde{p}\,=\widetilde{p}_n\,=\theta_n n^{-\alpha}$ with $\alpha\in [1,2)$ and $\theta_n \in (0, n^{\alpha})$
such that $\theta_n \asymp 1$.
This is precisely the case considered in Theorem 1.3 of \cite{Gold2013}. 
 It follows
 from Chapter 6.3 in \cite{J2000}  that
\bestar
\sigma_{n,d}^2\asymp \left\{
\begin{array}{ll}
n^{2-\alpha}, & \text{if} ~  d=0, \alpha\in [1,2),\\
n^{d+1-\alpha d},  & \text{if}~ d\in \mathbb{N}, \alpha \in [1, 1+1/d).
\end{array}
\right. 
\eestar
Hence, by  Proposition \ref{proppercolation}, we obtain
\be
d_K \Big(\frac{N_{n,d}-\mu_{n,d}}{\sigma_{n,d}}, \mathbf{Z}\Big)
=\left\{
\begin{array}{ll}
O(n^{-1+\alpha/2}),   & \text{if}~ d=0, \alpha \in [1,2), \\
O(n^{-\alpha/2+d(\alpha-1)}),   & \text{if}~ d\in \mathbb{N}, \alpha \in \left[1,\frac{2d}{2d-1}\right).
\end{array}
\right. \label{bounda2}
\ee

We recall the corresponding results from \cite{Gold2013} and \cite{K2017}:   Theorem 2.1 in \cite{Gold2013} yields
\be
d_K \Big(\frac{N_{n,d}-\mu_{n,d}}{\sigma_{n,d}}, \mathbf{Z}\Big)
=O(n^{-1/2+(\alpha-1)d/2}), \label{bounda2a2}
\ee
while Theorem 1.3 in \cite{K2017}   gives
\be
d_K \Big(\frac{N_{n,d}-\mu_{n,d}}{\sigma_{n,d}}, \mathbf{Z}\Big)
=\left\{
\begin{array}{ll}
O(n^{-1+\alpha/2}),   & \text{if}~d=0, \alpha \in [1,2), \\
O(n^{1/2-3d/2-\alpha+3\alpha d/2}),   & \text{if}~ d\in \mathbb{N}, \alpha \in \left[1,~\frac{3d-1}{3d-2}\right).
\end{array}
\right. \label{bounda2a1}
\ee

The bound in (\ref{bounda2}) is better than that in (\ref{bounda2a1})  for
 $d\ge 2$ and $\alpha\in (1, 2d/(2d-1))$,
since in this case
\bestar
-\alpha/2+d(\alpha-1)-(1/2-3d/2-\alpha+3\alpha d/2)
=\frac{1}{2}(\alpha-1)(1-d)<0.
\eestar
For $d=0$ and $\alpha\in (1,2)$, (\ref{bounda2}) is better than   (\ref{bounda2a2}).
One may note that  the result  (\ref{bounda2}) is weaker than (\ref{bounda2a2}) when $d\ge 2$ and $\alpha\in (1, 1+1/d)$. However, in this paper,  Proposition \ref{proppercolation} is sufficient, since together with  Theorem \ref{thBE1} it yields the rate 
$O(n^{-1/2})$
 for  Proposition \ref{lemmanu1nBE} below, a key result required in the proof  of  Theorem  \ref{th4}.
\end{rem}

The second proposition provides a Berry-Esseen bound for $\mu(\mathbb{T}_n)$, defined   in (\ref{munT}).

\begin{prop} \label{propa41} Let $\mu(\mathbb{T}_n)$  and  $\sigma_{1}^2$ be defined   in (\ref{munT}) and (\ref{sigma1}), respectively.
 Then we have
\be 
d_K\Big(\frac{\mu(\mathbb{T}_n)-\frac{np}{2-p}}{\sigma_{1}\sqrt{n}}, ~\mathbf{Z}\Big)\le Cn^{-1/2}.  \label{BE2.2}
\ee
\end{prop}

Applying Propositions  \ref{proppercolation} and \ref{propa41}, along with the conditioning and comparison techniques used in
the proof of the main results, we can establish the following Berry-Esseen theorem for $\nu_1(n)$. This result plays an important role in the proof of  Theorem  \ref{th4}.

\begin{prop} \label{lemmanu1nBE} Let $\nu_1(n)$ be defined in (\ref{nuk}), and set
\be
\sigma_{3}^2=\frac{2p(1-p)(3-p)}{(3-2p)(2-p)^2}.  \label{sigma3}
\ee
Then we have
 \be
d_K\Big(\frac{\nu_1(n)-\frac{np}{2-p}}{\sigma_{3}\sqrt{n}}, ~\mathbf{Z}\Big)\le Cn^{-1/2}. \label{nu1BE}
\ee
\end{prop}

The proofs of Theorem \ref{thBE1} and  Propositions  \ref{proppercolation}$-$\ref{lemmanu1nBE} are postponed to Section \ref{secta4}.

\section{Proofs of the main results} \label{secta2}
\subsection{Proof of Theorem \ref{th3}}

 Recall that $\hat{S}_n=\sum_{j=1}^n N_j(n) X_j$, where $N_j(n)$ is defined in  (\ref{defNjn}).
Define
\be
\mathscr{G}_1=\{\emptyset, \Omega\},~~~~
 \mathscr{G}_n=\sigma(\varepsilon_j, U_j: j=2,\cdots, n), ~n\ge 2,   \label{sigmaG}
\ee
where $\emptyset$ represents the empty set and $\Omega$ is the sample space.
Then   $\mathscr{G}_n$
is independent of $\sigma\{X_j, j=1, 2,\cdots\}$ and
$N_j(n) \in \mathscr{G}_n$ for any $1\le j\le n$ and $  n\in \mathbb{N}$. 
Therefore, conditional on  $\mathscr{G}_n$,  $\hat{S}_n$ can be treated as a sum of independent random variables. 

Note that
$
\mathbb{E}(\hat{S}_n|\mathscr{G}_n)=m_1 \sum_{j=1}^n N_j(n)=m_1 n
$
and
\bestar
\hat{B}_n^2:=\mbox{Var}(\hat{S}_n|\mathscr{G}_n)=\sigma_0^2 \sum_{j=1}^n N_j^2(n)=\sigma_0^2\sum_{k=1}^n k^2\nu_k(n)=\sigma_0^2Z_2(n),
\eestar
where  $
Z_{l}(n)=\sum_{k=1}^n k^l \nu_k(n)
$ for $l\ge 0$, and  $\nu_k(n)$  is defined in   (\ref{nuk}).
By applying the classical Berry-Esseen theorem (see, for instance, Theorem 5.4 in \cite{Petrov1995}), we have
\be
\sup_{x\in \mathbb{R}}\Big|\mathbb{P}\Big(\hat{B}_n^{-1} (\hat{S}_n-m_1 n)\le x \Big|\mathscr{G}_n\Big)-\Phi(x)\Big|\le C \frac{\hat{A}_{n}}{\hat{B}_n^{3}},  \label{diffhats11}
\ee
where
$
\hat{A}_n= \sum_{j=1}^n N_j^3(n)\mathbb{E}(|X_j|^{3})=\mathbb{E}(|X_1|^{3})Z_3(n)
$, and $\Phi(x)$ is the standard normal distribution function. Since the left-hand side of (\ref{diffhats11}) is trivially bounded by 1, we obtain the following refined bound:
\be
\sup_{x\in \mathbb{R}}\Big|\mathbb{P}\Big(\hat{B}_n^{-1} (\hat{S}_n-m_1 n)\le x \Big|\mathscr{G}_n\Big)-\Phi(x)\Big|\le C \Big(\frac{\hat{A}_{n}}{\hat{B}_n^{3}}\wedge 1\Big).   \label{diffhats1}
\ee

The remainder of the proof proceeds in two steps.

{\bf Step 1}: By using (\ref{diffhats1})  along with some elementary calculations, we can get that
 \be 
\sup_{x\in \mathbb{R}}\Big|\mathbb{P}\Big(\frac{\hat{S}_n-m_1 n}{\sigma_0\sqrt{b_n}}\le x \Big)-\mathbb{E}\Big(\Phi\Big( \frac{\sigma_0\sqrt{b_n} x}{\hat{B}_n}\Big) \Big)\Big| \le C\delta_{1,n},  \label{step1}
\ee 
where $b_n$ and $\delta_{1,n}$ are defined in (\ref{bn2}) and (\ref{delta1n}), respectively.

\begin{proof}[Proof of (\ref{step1})]
It follows from (\ref{diffhats1}) that
\bestar
&&\sup_{x\in \mathbb{R}}\Big|\mathbb{P}\Big(\frac{\hat{S}_n-m_1 n}{\sigma_0 \sqrt{b_n}}\le x \Big|\mathscr{G}_n\Big)-\Phi\Big( \frac{\sigma_0 \sqrt{b_n} x}{\hat{B}_n} \Big)\Big|\\
&\le& \sup_{x\in \mathbb{R}}\Big|\mathbb{P}\Big(\frac{\hat{S}_n-m_1 n}{\hat{B}_n}\le \frac{\sigma_0\sqrt{b_n} x}{\hat{B}_n} \Big|\mathscr{G}_n\Big)
-\Phi\Big( \frac{\sigma_0\sqrt{b_n} x}{\hat{B}_n} \Big)\Big|\le C\Big(\frac{\hat{A}_n}{\hat{B}_n^3}\wedge 1\Big),
\eestar
and consequently
\be 
\sup_{x\in \mathbb{R}}\Big|\mathbb{P}\Big(\frac{\hat{S}_n-m_1 n}{\sigma_0\sqrt{b_n}}\le x \Big)-\mathbb{E}\Big(\Phi\Big( \frac{\sigma_0\sqrt{b_n} x}{\hat{B}_n}\Big) \Big)\Big| \le C\mathbb{E} \Big(\frac{\hat{A}_n}{\hat{B}_n^{3}}\wedge 1\Big).  \label{BEa1}
\ee
Applying Lemma \ref{lemma38} gives
\be 
\mathbb{E} \Big(\frac{\hat{A}_n}{\hat{B}_n^3}\wedge 1\Big) &\le & C \mathbb{E} \Big(\frac{Z_3(n)}{(Z_2(n))^{3/2}}\wedge 1\Big)\nonumber\\
&\le& C\mathbb{P} (Z_2(n)\le (1/2)b_n)+C\mathbb{E} \Big(\frac{Z_3(n)}{(Z_2(n))^{3/2}}I(Z_2(n)>(1/2) b_n)\Big)\nonumber\\
&\le&  C\frac{\mathbb{E}(Z_2(n)-b_n)^2}{b_n^2}+ C \frac{b_3(n)}{b_n^{3/2}}
\le C\delta_{1,n},  \label{BEa2}
\ee 
where we have used the inequality (by (\ref{EZ2na}))
\be
\frac{\mathbb{E}(Z_{2}(n)-b_n)^2}{b_n^{2}}\le \frac{2(\mathbb{E}(Z_{2}(n))-b_n)^2+2\mbox{Var}(Z_2(n))}{b_n^{2}}\le Cn^{-2}+Cb_4(n)/b_n^2 \le C\delta_{1,n},  \label{Taylor3}
\ee
 $b_3(n)$ and $b_4(n)$ are defined in  (\ref{bln}). Now (\ref{step1}) follows  from 
(\ref{BEa1}) and (\ref{BEa2}).
\end{proof}

\vskip 0.3cm

{\bf Step 2}:  We will show that
\be
\sup_{x\in \mathbb{R}}\Big|\mathbb{E}\Big(\Phi\Big( \frac{\sigma_0\sqrt{b_n} x}{\hat{B}_n}\Big)\Big)-\Phi(x)\Big|\le C\delta_{1,n}.   \label{phihat}
\ee

\begin{proof}[Proof of (\ref{phihat})]
By Taylor's formula, we have
\be
\Phi\Big( \frac{\sigma_0\sqrt{b_n} x}{\hat{B}_n}\Big)-\Phi(x)
= x\phi(x) \Big(\frac{\sigma_0\sqrt{b_n}}{\hat{B}_n}-1\Big)+\frac{1}{2}x^2\phi'(\zeta_n x) \Big(\frac{\sigma_0\sqrt{b_n}}{\hat{B}_n}-1\Big)^2,  \label{Taylor1}
\ee
where $\phi(x)$ is the standard normal density function and $ (\sigma_0\sqrt{b_n}/\hat{B}_n)\wedge 1\le \zeta_n\le  (\sigma_0\sqrt{b_n}/\hat{B}_n)\vee 1$.
Let   $E_n=\{\hat{B}_n^2 > (1/2)\sigma_0^2 b_n \}$ and note that for any $x, y>0$, 
\be
\frac{x}{y}-1&=&\frac{x^2-y^2}{y(x+y)}=\frac{x^2-y^2}{2x^2}+\frac{(x^2-y^2)(2x^2-xy-y^2)}{2x^2y(x+y)}\nonumber\\
&=&\frac{x^2-y^2}{2x^2}+\frac{(x^2-y^2)^2}{2x^2y(x+y)}+\frac{(x^2-y^2)^2}{2xy(x+y)^2}.  \label{xy}
\ee
By taking $x=\sigma_0\sqrt{b_n}$ and $y=\hat{B}_n$ in (\ref{xy}) and
 applying (\ref{EZ2na}) and
(\ref{Taylor3}), we have
\be
\Big|\mathbb{E}\Big(\Big(\frac{\sigma_0\sqrt{b_n}}{\hat{B}_n}-1\Big)I_{E_n}\Big)\Big|&\le& 
\Big|\frac{\mathbb{E}(\hat{B}_n^2)-\sigma_0^2b_n}{2\sigma_0^2b_n}\Big|+\mathbb{E}\Big|\frac{\hat{B}_n^2-\sigma_0^2b_n}{2\sigma_0^2b_n} I_{E_n^c}\Big|+
\mathbb{E}\Big(\frac{(\hat{B}_n^2-\sigma_0^2b_n)^2}{\sigma_0^3b_n^{3/2}\hat{B}_n}I_{E_n}\Big)\nonumber\\
&\le& 
\Big|\frac{\mathbb{E}(\hat{B}_n^2)-\sigma_0^2b_n}{2\sigma_0^2b_n}\Big|+\frac{1}{2} \mathbb{P}(\hat{B}_n^2\le (1/2)\sigma_0^2 b_n)+
\sqrt{2}\mathbb{E}\Big(\frac{(\hat{B}_n^2-\sigma_0^2b_n)^2}{\sigma_0^4b_n^2}\Big)\nonumber\\
&\le& 
\Big|\frac{\mathbb{E}(\hat{B}_n^2)-\sigma_0^2b_n}{2\sigma_0^2b_n}\Big|+\frac{C\mathbb{E}(\hat{B}_n^2-\sigma_0^2b_n)^2}{\sigma_0^4b_n^2}\nonumber\\
&=& \Big|\frac{\mathbb{E}(Z_2(n))-b_n}{2b_n}\Big|+\frac{C\mathbb{E}(Z_{2}(n)-b_n)^2}{b_n^{2}}\nonumber\\
&\le& Cn^{-1}+C\delta_{1,n}\le C\delta_{1,n}.  \label{Taylor2}
\ee
Similarly,  
\bestar
\mathbb{E}\Big(\Big(\frac{\sigma_0\sqrt{b_n}}{\hat{B}_n}-1\Big)^2I_{E_n}\Big)&\le&
2\mathbb{E}\Big(\frac{\hat{B}_n^2}{\sigma_0^2 b_n}\Big(\frac{\sigma_0\sqrt{b_n}}{\hat{B}_n}-1\Big)^2I_{E_n}\Big)\\
&\le&  2\mathbb{E}\Big(\frac{\hat{B}_n}{\sigma_0\sqrt{b_n}}-1\Big)^2
\le \frac{2\mathbb{E}(\hat{B}_n^2-\sigma_0^2b_n)^2}{\sigma_0^4b_n^2}\le C\delta_{1,n}.
\eestar
Therefore,
\be
\mathbb{E} \Big(x^2 |\phi'(\zeta_n x) |\Big(\frac{\sigma_0\sqrt{b_n}}{\hat{B}_n}-1\Big)^2I_{E_n}\Big)
&\le&  c_0\mathbb{E} \Big(\zeta_n^{-2}\Big(\frac{\sigma_0\sqrt{b_n}}{\hat{B}_n}-1\Big)^2I_{E_n}\Big)\nonumber\\
&\le& c_0\mathbb{E} \Big(\Big(1+\frac{\hat{B}_n^2}{\sigma_0^2 b_n}\Big)\Big(\frac{\sigma_0\sqrt{b_n}}{\hat{B}_n}-1\Big)^2I_{E_n}\Big)\le C\delta_{1,n},\qquad
\label{Taylora1}
\ee
where $c_0=\sup_{x} x^2|\phi'(x)|<\infty$.
By using (\ref{Taylor1}),  (\ref{Taylor2}),  (\ref{Taylora1}), and the fact that $\sup_{x\in \mathbb{R}} |x\phi(x)|<\infty$,  we have
\be
\sup_{x\in \mathbb{R}}\Big|\mathbb{E}\Big(\Big(\Phi\Big( \frac{\sigma_0\sqrt{b_n} x}{\hat{B}_n}\Big)-\Phi(x)\Big)I_{E_n}\Big)\Big|\le C\delta_{1,n}. \label{diffphi}
\ee
Observe that by (\ref{Taylor3}),
\bestar
\sup_{x\in \mathbb{R}}\Big|\mathbb{E}\Big(\Big(\Phi\Big( \frac{\sigma_0\sqrt{b_n} x}{\hat{B}_n}\Big)-\Phi(x)\Big)I_{E_n}^c\Big)\Big|\le \mathbb{P}(E_n^c)
\le \frac{\mathbb{E}(\hat{B}_n^2-\sigma_0^2b_n)^2}{\sigma_0^4b_n^2}
=\frac{\mathbb{E}(Z_{2}(n)-b_n)^2}{b_n^{2}}\le C\delta_{1,n}.
\eestar
 This, together with (\ref{diffphi}), proves (\ref{phihat}).
\end{proof}

Combining   (\ref{step1}) and  (\ref{phihat}) immediately yields the desired results. 
The proof of Theorem \ref{th3} is complete.

\subsection{Proof of Theorem \ref{th4}}

We adopt the same approach as in the proof of Theorem \ref{th3}. However, a key difference 
arises from the fact that the weights in $\hat{S}_n=\sum_{j=1}^n N_j(n) X_j$ sum to a fixed value
($\sum_{j=1}^n N_j(n)= n$), while those in
 $\check{S}_n=\sum_{j=1}^{n} \Delta(T_j(n)) X_j$ do not. This distinction complicates the proof. We proceed in three steps, with the first two similar to those in the proof of Theorem \ref{th3}.

{\bf Step 1}: We will prove that
\be
\sup_{x\in \mathbb{R}}\Big|\mathbb{P}\Big(\frac{\check{S}_n-\check{b} n}{\check{\sigma} \sqrt{n}}\le x \Big)-\mathbb{E}\Big(\Phi\Big( \frac{\check{\sigma} x\sqrt{n} -m_1(\nu_1(n)-\frac{np}{2-p}) }{\check{B}_n}\Big) \Big)\Big| \le C\delta_{2,n}\label{addsup}
\ee
with $$\check{B}_n^2=\frac{m_2n}{3}-\frac{2m_2\nu_2(n)}{3}+\Big(\frac{2m_2}{3}-m_1^2\Big)\nu_1(n),$$ where $\delta_{2,n}$ and
 $\nu_k(n)$ (for $k=1,2$)
 are defined in (\ref{delta2n}) and (\ref{nuk}), respectively. 

\begin{proof}[Proof of (\ref{addsup})]
Observe that by (\ref{checksnsum}),
\bestar
\check{S}_n=\sum_{j=1}^{i(n)} \Delta(T_j(n)) X_j= \sum_{k=1}^n \check{S}_k(n),
\eestar where
\bestar
\check{S}_k(n)=\sum_{j=1}^{i(n)} \Delta(T_j(n)) X_j I(N_j(n)=k),~~~~k\ge 1,
\eestar
and $i(n)$ and $N_j(n)$ are defined in (\ref{in}) and (\ref{defNjn}), respectively.
  For each $k\ge 1$,
let $(Y_k(n))_{n\ge 1}$ be a sequence of i.i.d. copies of $\Delta(\mathbb{T}_k)X_1$, where $\mathbb{T}_k$ is a random recursive tree of size $k$ that is independent of $X_1$.
We also assume that these sequences are mutually independent and independent of $\{U_i, \varepsilon_i\}_{i\ge 2}$.
Define $S_k(0)=0$ and $S_k(n) = Y_k(1) + \cdots+ Y_k(n)$ for any $n\ge 1$. It follows  from the proof of Lemma 4.2 in \cite{Bertoin2024} that
$(\check{S}_k(n))_{k\ge 1} \stackrel{d}{=} (S_k(\nu_k(n)))_{k\ge 1},
$
where $(\nu_k(n))_{k\ge 1}$ is defined in (\ref{nuk}) and is indepent of  $(S_k(\cdot))_{k\ge 1}$.
Therefore,
\bestar
\check{S}_n \stackrel{d}{=} \sum_{k=1}^{n} S_k(\nu_k(n)).\label{checksn}
\eestar

Using the definition of $Y_k(n)$ and applying Corollary 2.3  in  \cite{Bertoin2024} shows that  $\mathbb{P}(Y_2(n)=0)=1$,
\bestar
&&\mathbb{E}(Y_1(n))=m_1, ~~\mbox{Var}(Y_1(n))=\sigma_0^2,~~\mathbb{E}|Y_1(n)|^3=\mathbb{E}(|X_1|^3),\\
&&\mathbb{E}(Y_k(n))=0, ~~\mbox{Var}(Y_k(n))=km_2/3,~~\mathbb{E}|Y_k(n)|^3\le 4k^{3/2}\mathbb{E}(|X_1|^3),~~~k\ge 3,
\eestar
where we have used the inequality
\bestar
\mathbb{E}(|\Delta(\mathbb{T}_k)|^3)\le \big(\mathbb{E}(|\Delta(\mathbb{T}_k)|^4)\big)^{3/4}\le (6k^2)^{3/4}\le 4k^{3/2}, ~~k\ge 3.
\eestar
By arguments similar to those in the proof of (\ref{diffhats1}), we obtain
\be
\sup_{x\in \mathbb{R}}\Big|\mathbb{P}\Big(\check{B}_n^{-1} \Big(\sum_{k=1}^{n} S_k(\nu_k(n))-m_1\nu_1(n)\Big)\le x \Big|\mathscr{G}_n\Big)-\Phi(x)\Big|\le C \Big(\frac{\check{A}_{n}}{\check{B}_n^{3}}\wedge 1\Big),
\ee
where 
\bestar
\check{B}_n^2&=& \frac{m_2}{3}\sum_{k=3}^n k\nu_k(n)+\sigma_0^2\nu_1(n),\\
\check{A}_n&=& \sum_{k=1}^n \nu_k(n) \mathbb{E}|Y_k(n)|^{3}\le 4\mathbb{E}(|X_1|^{3}) Z_{3/2}(n),
\eestar
and $Z_{3/2}(n)$ and $\mathscr{G}_n$ are defined in Lemma \ref{lemma38} and  (\ref{sigmaG}), respectively.
Hence
\bestar
&&\sup_{x\in \mathbb{R}}\Big|\mathbb{P}\Big(\frac{\sum_{k=1}^{n} S_k(\nu_k(n))-\check{b} n}{\check{\sigma}\sqrt{n}}\le x \Big|\mathscr{G}_n\Big)-\Phi\Big( \frac{\check{\sigma} x\sqrt{n}-m_1(\nu_1(n)-\frac{np}{2-p}) }{\check{B}_n} \Big)\Big|\\
&\le& \sup_{x\in \mathbb{R}}\Big|\mathbb{P}\Big(\frac{\sum_{k=1}^{n} S_k(\nu_k(n))-m_1\nu_1(n)}{\check{B}_n}\le \frac{\check{\sigma} x\sqrt{n}-m_1(\nu_1(n)-\frac{np}{2-p}) }{\check{B}_n} \Big|\mathscr{G}_n\Big)\\
&&~~~~~~-\Phi\Big( \frac{\check{\sigma} x\sqrt{n}-m_1(\nu_1(n)-\frac{np}{2-p}) }{\check{B}_n} \Big)\Big|\le C\Big(\frac{\check{A}_n}{\check{B}_n^3}\wedge 1\Big)
\eestar
and consequently,
\be 
\sup_{x\in \mathbb{R}}\Big|\mathbb{P}\Big(\frac{\check{S}_n-\check{b} n}{\check{\sigma} \sqrt{n}}\le x \Big)-\mathbb{E}\Big(\Phi\Big( \frac{\check{\sigma} x\sqrt{n} -m_1(\nu_1(n)-\frac{np}{2-p}) }{\check{B}_n}\Big) \Big)\Big| \le C\mathbb{E} \Big(\frac{\check{A}_n}{\check{B}_n^3}\wedge 1\Big). 
\label{BEa21}
\ee 
Since  $\sum_{k=1}^n k\nu_k(n)=n$ and $\sigma_0^2=m_2-m_1^2$, we have
\bestar
\check{B}_n^2=\frac{m_2n}{3}-\frac{2m_2\nu_2(n)}{3}+\Big(\frac{2m_2}{3}-m_1^2\Big)\nu_1(n).
\eestar
Define
\be 
\sigma_{4}^2=\frac{m_2}{3}-\frac{2p(1-p)m_2}{3(2-p)(3-2p)}+\frac{p(2m_2/3-m_1^2)}{2-p}=\frac{m_2}{3-2p}-\frac{pm_1^2}{2-p}.\label{sigma4add}
\ee
Applying Lemma  \ref{lemmaadd2} gives that $\mathbb{E}(\check{B}_n^2)=\sigma_{4}^2n+O(1)$ and
\bestar
\mbox{Var}(\check{B}_n^2)&=&\mbox{Var}\Big(\Big(\frac{2m_2}{3}-m_1^2\Big)\nu_1(n)-\frac{2m_2\nu_2(n)}{3}\Big)\\
&\le& C\mbox{Var}(\nu_1(n))+C\mbox{Var}(\nu_2(n))\le Cn.
\eestar
It follows that
\be
\frac{\mathbb{E}|\check{B}_n^2-\sigma_{4}^2n|}{\sigma_{4}^2n}&=&\frac{\mathbb{E}|\check{B}_n^2-\mathbb{E}(\check{B}_n^2)|+|\mathbb{E}(\check{B}_n^2)-\sigma_{4}^2n|}{\sigma_{4}^2n}\nonumber\\
&\le& \frac{\sqrt{\mbox{Var}(\check{B}_n^2)}+|\mathbb{E}(\check{B}_n^2)-\sigma_{4}^2n|}{\sigma_{4}^2n}\le Cn^{-1/2}.  \label{diffadd1}
\ee
This, together with Lemma \ref{lemma38}, implies that
\bestar
\mathbb{E} \Big(\frac{\check{A}_n}{\check{B}_n^3}\wedge 1\Big) &\le & C \mathbb{E} \Big(\frac{Z_{3/2}(n)}{\check{B}_n^3}\wedge 1\Big)\nonumber\\
&\le& C\Big( \mathbb{P} \Big(\check{B}_n\le \frac{1}{2} \sigma_{4}\sqrt{n}\Big)+\mathbb{E} \Big(\frac{Z_{3/2}(n)}{\check{B}_n^3}I\Big(\check{B}_n> \frac{1}{2} \sigma_{4}\sqrt{n}\Big)\Big)\Big)\nonumber\\
&\le& \frac{C\mathbb{E}|\check{B}_n^2-\sigma_{4}^2n|}{\sigma_{4}^2n}+ C \mathbb{E} \Big(\frac{Z_{3/2}(n)}{n^{3/2}}\Big)\\
&\le& Cn^{-1/2}+Cn^{-3/2}b_{3/2}(n)\le C\delta_{2,n}.
\eestar
Hence, the result (\ref{addsup}) follows from (\ref{BEa21}).
\end{proof}

{\bf Step 2}: We have
\be
&&\sup_{x\in \mathbb{R}}\Big|\mathbb{E}\Big(\Phi\Big( \frac{\check{\sigma} x\sqrt{n} -m_1(\nu_1(n)-\frac{np}{2-p}) }{\check{B}_n}\Big) \Big)-
\mathbb{E}\Big(\Phi\Big( \frac{\check{\sigma} x\sqrt{n}-m_1(\nu_1(n)-\frac{np}{2-p}) }{\sigma_{4}\sqrt{n}}\Big) \Big)\Big|\nonumber\\
&&~~~~~~~~~~~~~~~~\le  Cn^{-1/2}, \label{diffadd2}
\ee
where $\sigma_{4}^2$ is defined in (\ref{sigma4add}).

\begin{proof}[Proof of (\ref{diffadd2})]
For any $a>0, b>0$ and $x\in \mathbb{R}$, by the mean value theorem,  we have
\bestar
|\Phi(ax)-\Phi(bx)|=\frac{1}{\sqrt{2\pi}}|(a-b)x|e^{-\zeta^2/2}\le \frac{1}{\sqrt{2\pi}}|(a-b)x|e^{-(a^2\wedge b^2)x^2/2 }\le \frac{c_1|a-b|}{a\wedge b},
\eestar
where $c_1=\frac{1}{\sqrt{2\pi}} \sup\limits_{x\ge 0} xe^{-x^2/2}$ and $\zeta$ lies between $ax$ and $bx$.
Then by (\ref{diffadd1}), 
\bestar
&&\sup_{x\in \mathbb{R}}\Big|\mathbb{E}\Big(\Phi\Big( \frac{\check{\sigma} x\sqrt{n} -m_1(\nu_1(n)-\frac{np}{2-p}) }{\check{B}_n}\Big) \Big)-
\mathbb{E}\Big(\Phi\Big( \frac{\check{\sigma} x\sqrt{n}-m_1(\nu_1(n)-\frac{np}{2-p}) }{\sigma_{4}\sqrt{n}}\Big) \Big)\Big|\nonumber\\
&&~~~~~~~~~~~~~~~~\le \mathbb{P}(\check{B}_n^2\le (1/2)\sigma_4^2 n)+c_1\mathbb{E}\Big(\frac{|\check{B}_n-\sigma_4\sqrt{n} |}{\check{B}_n\wedge (\sigma_4\sqrt{n})}I(\check{B}_n^2>(1/2)\sigma_4^2 n)\Big)\nonumber\\
&&~~~~~~~~~~~~~~~~\le C\frac{\mathbb{E}|\check{B}_n^2-\sigma_{4}^2n|}{\sigma_{4}^2n}\le Cn^{-1/2}.  
\eestar
This proves (\ref{diffadd2}).
\end{proof}

{\bf Step 3}:
By Proposition \ref{lemmanu1nBE}, we have
\be
d_K\Big(\frac{\nu_1(n)-\frac{np}{2-p}}{\sqrt{n}}, ~\mathbf{Z}'\Big)\le Cn^{-1/2},  \label{BEnua}
\ee
where $\mathbf{Z}' $ is a normal random variable with mean $0$ and variance $\sigma_3^2$,
and  $\sigma_3^2$ is defined in (\ref{sigma3}).
Furthermore, we can show that
\be
&& \sup_{x\in \mathbb{R}}\Big|
\mathbb{E}\Big(\Phi\Big( \frac{\check{\sigma} x \sqrt{n}-m_1(\nu_1(n)-\frac{np}{2-p}) }{\sigma_{4} \sqrt{n}}\Big) \Big)-
\mathbb{E}\Big(\Phi\Big(\frac{\check{\sigma} x-m_1 \mathbf{Z}'}{\sigma_{4}}\Big)\Big)\Big|\le Cn^{-1/2}.\label{diffadd3}
\ee

\vskip 0.3cm

\begin{proof}[Proof of (\ref{diffadd3})]
  The conclusion is obviously true when $m_1=0$.
In the following, we assume that $m_1\ne 0$. For each
 $x\in \mathbb{R}$,  we define a function $f_x$ by
$
f_{x}(t)=\frac{\check{\sigma} x-m_1 t }{\sigma_{4}}, t\in \mathbb{R}$.
Then $f_x$ is a strictly monotonic and continuous function on $\mathbb{R}$.
Note that
\bestar
\mathbb{E}(\Phi(f_x(Y)))=\frac{1}{2}-\int_{-\infty}^{0} \phi(t) \mathbb{P}(f_x(Y) \le t) dt+\int_0^{\infty}\phi(t) \mathbb{P} (f_x(Y)>t) dt
\eestar
holds for any random variable $Y$,
where $\phi(t)$ is the standard normal density function.
Let $ f_x^{-1}$ be the inverse function of $f_x$. 
If $m_1<0$, then
\bestar
\mathbb{E}(\Phi(f_x(Y)))=\frac{1}{2}-\int_{-\infty}^{0} \phi(t) \mathbb{P}(Y\le f_x^{-1}(t)) dt+\int_0^{\infty}\phi(t) \mathbb{P} (Y>f_x^{-1}(t)) dt.
\eestar
If $m_1>0$, then
\bestar
\mathbb{E}(\Phi(f_x(Y)))=\frac{1}{2}-\int_{-\infty}^{0} \phi(t) \mathbb{P}(Y>f_x^{-1}(t)) dt+\int_0^{\infty}\phi(t) \mathbb{P} (Y\le f_x^{-1}(t)) dt.
\eestar
In both cases, it follows from   (\ref{BEnua}) that
\bestar
&& \sup_{x\in \mathbb{R}}\Big|
\mathbb{E}\Big(\Phi\Big( \frac{\check{\sigma} x \sqrt{n}-m_1(\nu_1(n)-\frac{np}{2-p}) }{\sigma_{4} \sqrt{n}}\Big) \Big)-
\mathbb{E}\Big(\Phi\Big(\frac{\check{\sigma} x-m_1 \mathbf{Z}'}{\sigma_{4}}\Big)\Big)\Big|\nonumber\\
&=& \sup_{x\in \mathbb{R}}\Big|
\mathbb{E}\Big(\Phi\Big(f_x\Big(\frac{\nu_1(n)-\frac{np}{2-p}}{\sqrt{n}}\Big)\Big)\Big)-
\mathbb{E}(\Phi(f_x(\mathbf{Z}')))\Big| \le Cn^{-1/2}.
\eestar
This proves (\ref{diffadd3}).
\end{proof}

The desired result, Theorem \ref{th4}, now follows from (\ref{addsup}), (\ref{diffadd2}) and (\ref{diffadd3})  since
\bestar
\mathbb{E}\Big(\Phi\Big(\frac{\check{\sigma} x-m_1 \mathbf{Z}'}{\sigma_{4}}\Big)\Big)&=&\mathbb{P}\Big(\mathbf{Z}\le \frac{\check{\sigma} x-m_1 \mathbf{Z}'}{\sigma_{4}}\Big)
=\mathbb{P}\Big(\frac{\sigma_{4}\mathbf{Z}+m_1 \mathbf{Z}'}{\check{\sigma}}\le x\Big)\nonumber\\
&=& \mathbb{P}\Big(\frac{\sqrt{\sigma_{4}^2+m_1^2 \sigma_{3}^2}}{\check{\sigma}}\mathbf{Z}\le x\Big)=\Phi(x), \label{normalcal}
\eestar
where $\mathbf{Z}$ is a standard normal random variable  and  independent of $\mathbf{Z}'$, and we have used the fact that $\sigma_{4}^2+m_1^2 \sigma_{3} ^2=\check{\sigma}^2$. The proof of Theorem \ref{th4} is complete.

\section{Proofs of Theorem \ref{thBE1} and  Propositions  \ref{proppercolation}$-$\ref{lemmanu1nBE}} \label{secta4}

\subsection{Proof of Theorem \ref{thBE1}} \label{subsect4.1}

Without loss of generality, we assume that $\sigma^2=\mbox{Var}(f(\xi))=1$  throughout Section \ref{subsect4.1}.

Before proceeding to the proof, we make some preparations.
Let $\mathscr{F}_{0}=\sigma(\emptyset, \Omega)$ and $\mathscr{F}_{k}=\sigma(\xi_1, \cdots, \xi_k)$ for $1\le k\le n$.  Similarly, define $\mathscr{F}_{k}'$ from $\xi'$ for $0\le k\le n$. For each $1\le j\le n$,  set
	\be 
	Y_j=\mathbb{E}(f(\xi)|\mathscr{F}_j)-\mathbb{E}(f(\xi)|\mathscr{F}_{j-1}).
\label{Yja}
	\ee 
Then $\{(Y_j, \mathscr{F}_{j}):  1\le j\le n\}$ is a martingale difference sequence, and
\be 
W=f(\xi)-\mathbb{E} (f(\xi))=\sum_{j=1}^n Y_j. \label{Wsum}
\ee 
Since $Y_j\in \mathscr{F}_j$,  there exists a measurable function $g_j$  on $\mathbb{R}^j$ such that $ Y_j=g_j(\xi_1, \cdots, \xi_j). $
For $1\le i\le j\le n$,  define
\bestar
	Y_j^{\{i\}}=g_j(\xi_1,\cdots, \xi_{i-1}, \xi_i',  \xi_{i+1},\cdots, \xi_{j}).
\eestar  
Furthermore, we can  rewrite $Y_j^{\{i\}}$ in terms of conditional expectations as follows:
\be 
&& Y_i^{\{i\}}=\mathbb{E}(f(\xi^{\{i\}})|\mathscr{F}_i^{\{i\}})-\mathbb{E}(f(\xi)|\mathscr{F}_{i-1}), \qquad 1\le i\le n,  \label{Yiia}\\
&& Y_j^{\{i\}}=\mathbb{E}(f(\xi^{\{i\}})|\mathscr{F}_j^{\{i\}})-\mathbb{E}(f(\xi^{\{i\}})|\mathscr{F}_{j-1}^{\{i\}}),\qquad 1\le i<j\le n, \label{Yija}
\ee
where   $\mathscr{F}_j^{\{i\}}=\sigma(\xi_1,\cdots, \xi_{i-1}, \xi_i',  \xi_{i+1},\cdots, \xi_{j})$ for $i\le j$, and $\xi^{\{i\}}$ is defined in \eqref{XA}.

In the proof of Theorem \ref{thBE1}, we need some basic results about $\{Y_j\}$ and $\{Y_j^{\{i\}}\}$, which are listed in Lemmas \ref{lemmaY1} and \ref{lemmaY2} below.

\begin{lemma}  	\label{lemmaY1}
Let $\{\Delta_i, 1\le i\le n\}$ and $\{\Delta_{ij}, 1\le i\ne j\le n\}$ be defined 
as in  Theorem \ref{thBE1}. Then the following hold:
\begin{enumerate} 
 \item[(a)] For any $1\le i\le n$,  $Y_i-Y_i^{\{i\}}=\Delta_i;$
\item[(b)]  For $i<j$,
	$
	Y_j-Y_j^{\{i\}}=\mathbb{E}(\Delta_{ij}|\mathscr{F}_{j-1}',\mathscr{F}_j);
$
	\item[(c)]  $\sum_{i=1}^n\mathbb{E}(Y_i^{\{i\}}|\mathscr{F}_n) =0$  and
$
\sum_{i=1}^n \mathbb{E}(Y_i-Y_i^{\{i\}})^2=2.
$
\end{enumerate}
\end{lemma}

\begin{proof}
For $1\le i\le n$,   recalling that  $\xi_i'$ is independent of  $\sigma(\xi,  \mathscr{F}_{i})=\mathscr{F}_{n}$,
we can apply the properties of conditional independence (see, for instance, Chapter 9 of \cite{Chung2001}) to obtain that  $\xi_i'$ and  $\xi$ are conditionally independent,
given $\mathscr{F}_{i}$, 
and consequently
$
\mathbb{E}(f(\xi)|\mathscr{F}_i)=\mathbb{E}(f(\xi)|\xi_i', \mathscr{F}_i).
$
Similarly, we have $
\mathbb{E}(f(\xi^{\{i\}})|\mathscr{F}_i^{\{i\}})=\mathbb{E}(f(\xi^{\{i\}})|\xi_i', \mathscr{F}_i).
$ Hence, it follows from \eqref{Yja} and \eqref{Yiia} that
	\be 
	Y_i-Y_i^{\{i\}}=\mathbb{E}(f(\xi)|\mathscr{F}_i)-\mathbb{E}(f(\xi^{\{i\}})|\mathscr{F}_i^{\{i\}})
	=\mathbb{E}(f(\xi)-f(\xi^{\{i\}})|\xi_i',\mathscr{F}_i)=\Delta_i,
\label{A4.20}
	\ee 
which proves (a).

Note that $\mathbb{E}(f(\xi)|\mathscr{F}_{j-1})=\mathbb{E}(f(\xi^{(j)})|\mathscr{F}_{j-1})$, since  $\xi$ and $\xi^{(j)}$ have the same conditional
distribution given $\mathscr{F}_{j-1}$.
	Using  similar arguments to those in the proof of (a) shows  that for $i<j$,
	\bestar
	Y_j&=&\mathbb{E}(f(\xi)|\mathscr{F}_j)-\mathbb{E}(f(\xi^{(j)})|\mathscr{F}_{j-1})\\
&=&\mathbb{E}(f(\xi)-f(\xi^{\{j\}})|\mathscr{F}_j)=\mathbb{E}(f(\xi)-f(\xi^{\{j\}})|\xi_i', \mathscr{F}_j),\\
	Y_j^{\{i\}}&=&\mathbb{E}(f(\xi^{\{i\}})-f(\xi^{\{i,j\}})|\mathscr{F}_j^{\{i\}})\\
&=&\mathbb{E}(f(\xi^{\{i\}})-f(\xi^{\{i,j\}})|\xi_i, \mathscr{F}_j^{\{i\}})
=\mathbb{E}(f(\xi^{\{i\}})-f(\xi^{\{i,j\}})|\xi_i', \mathscr{F}_j).
	\eestar
This implies that  for $i<j$,
	\bestar
	Y_j-Y_j^{\{i\}}&=&\mathbb{E}(f(\xi)-f(\xi^{\{j\}})-f(\xi^{\{i\}})+f(\xi^{\{i,j\}})|\xi_i',\mathscr{F}_j)\\
	&=&\mathbb{E}(\Delta_{ij}|\xi_i',\mathscr{F}_j)=\mathbb{E}(\Delta_{ij}|\mathscr{F}_{j-1}',\mathscr{F}_j).
	\eestar
Hence (b) holds.

By the conditional independence of $\{\xi_1, \cdots, \xi_{i-1}, \xi_i'\}$
and $\xi_i, \cdots, \xi_{n}$ given $\mathscr{F}_{i-1}$, we obtain
\bestar
	\mathbb{E}(Y_i^{\{i\}}|\mathscr{F}_n)&=&\mathbb{E}(g_i(\xi_1, \cdots, \xi_{i-1}, \xi_i')|\mathscr{F}_n)=
	\mathbb{E}(g_i(\xi_1, \cdots, \xi_{i-1}, \xi_i')|\mathscr{F}_{i-1})\nonumber\\
	&=& \mathbb{E}(g_i(\xi_1, \cdots, \xi_{i-1}, \xi_i)|\mathscr{F}_{i-1})=\mathbb{E}(Y_i|\mathscr{F}_{i-1})=0.
	\eestar
Hence $\sum_{i=1}^n\mathbb{E}(Y_i^{\{i\}}|\mathscr{F}_n) =0$.
Moreover,  by noting that  $Y_i\stackrel{d}{=} Y_i^{\{i\}}$,  we have
	\bestar
	\mathbb{E}(Y_i-Y_i^{\{i\}})^2=2\mathbb{E}(Y_i^2)-2\mathbb{E} (Y_i \mathbb{E}(Y_i^{\{i\}}|\mathscr{F}_{n}))=2\mathbb{E}(Y_i^2).    
	\eestar
Since $\{Y_i, 1\le i\le n\}$ is a martingale differece sequence, it follows from
(\ref{Wsum}) that
\bestar
 \sum_{i=1}^n \mathbb{E}(Y_i-Y_i^{\{i\}})^2= 2\sum_{i=1}^n \mathbb{E}(Y_i^2)=2\mbox{Var} (f(\xi))=2.
	\eestar
The proof of (c) is complete.
\end{proof}

\begin{lemma} \label{lemmaY2}
For  $1\le i\le j\le n$,  define 
\be 
D_{1ij}=(Y_i-Y_i^{\{i\}})(Y_j-Y_j^{\{i\}}),  \qquad D_{2ij}=|Y_i-Y_i^{\{i\}}|(Y_j-Y_j^{\{i\}}).  \label{D1ijD2ij}
\ee 
Then
\begin{enumerate}
\item[$(a)$] $\mathbb{E}(D_{2ii})=0$,  and
$\mathbb{E}(D_{lij})=0
$  for  $l=1,2$ and any $i<j$;

\item[$(b)$] $
\mathbb{E}(D_{lij}D_{li'j'})=0
$  for  $l=1,2$ and any  $i<j, i'<j'$ with $j\ne j'$;

\item[$(c)$] 
$
\mathbb{E}(D_{2ii}D_{2i'i'})=0
$ for any $i\ne i'$.
\end{enumerate}
\end{lemma}

\begin{proof}
Since $(Y_i, Y_i^{\{i\}})$ is an exchangeable pair, we have $$\mathbb{E}(D_{2ii})=\mathbb{E}(|Y_i-Y_i^{\{i\}}|(Y_i-Y_i^{\{i\}}))=0.$$
Now consider any $i<j$. By the conditional independence of $Y_j$
and $\xi_i'$ given $\mathscr{F}_{j-1}$,  we have
$\mathbb{E}(Y_j|\xi_i',\mathscr{F}_{j-1})=\mathbb{E}(Y_j|\mathscr{F}_{j-1})=0$.
Similarly, 
$
\mathbb{E}(Y_j^{\{i\}}|\xi_i',\mathscr{F}_{j-1})
=\mathbb{E}(Y_j^{\{i\}}|\xi_i,\mathscr{F}_{j-1}^{\{i\}})
=\mathbb{E}(Y_j^{\{i\}}|\mathscr{F}_{j-1}^{\{i\}})=0.
$
Consequently,
	\bestar
	\mathbb{E}(D_{1ij})&=&\mathbb{E}\Big((Y_i-Y_i^{\{i\}})\mathbb{E}(Y_j-Y_j^{\{i\}}|\xi_i',\mathscr{F}_{j-1})\Big)=0,  \\
	\mathbb{E}(D_{2ij})&=&\mathbb{E}\Big(|Y_i-Y_i^{\{i\}}|\mathbb{E}(Y_j-Y_j^{\{i\}}|\xi_i',\mathscr{F}_{j-1})\Big)=0.   
	\eestar
Hence (a) holds.

Next, take any $i<j, i'<j'$ with $j\ne j'$. Without loss of generality, assume
$j<j'$. Then, using similar arguments to those in the proof of (a), we obtain
\bestar
\mathbb{E}(D_{1ij}D_{1i'j'})
=\mathbb{E}\Big(D_{1ij}(Y_{i'}-Y_{i'}^{\{i'\}})\mathbb{E}(Y_{j'}-Y_{j'}^{\{i'\}}|\xi_i', \xi_{i'}',\mathscr{F}_{j'-1})\Big)=0.
\eestar
 Similarly,  $\mathbb{E}(D_{2ij}D_{2i'j'})=0$. The proof of (b) is complete.

For any $i\ne i'$, note that $(Y_i, Y_i^{\{i\}}, Y_{i'}, Y_{i'}^{\{i'\}})
\stackrel{d}{=} (Y_i, Y_i^{\{i\}}, Y_{i'}^{\{i'\}}, Y_{i'})$. Then
$\mathbb{E}(D_{2ii}D_{2i'i'})=-\mathbb{E}(D_{2ii}D_{2i'i'})$. Hence  $\mathbb{E}(D_{2ii}D_{2i'i'})=0$.
This proves (c).
\end{proof}

\vskip 0.3cm

\begin{proof}[Proof of Theorem \ref{thBE1}]
The proof consists of two steps.  We first give a preliminary upper bound (see \eqref{dKexpre} below) for $d_K(W, \mathbf{Z})$ using the exchangeable pair approach of Stein’s method. The idea originates from \cite{S2019} and \cite{S2021}. 
 In the second step, we further estimate the bound in \eqref{dKexpre}
in terms of the first-order conditional differences $\{\Delta_i, 1\le i\le n\}$ and  the second-order conditional differences $\{\Delta_{ij}, 1\le i\ne j \le n\}$.

{\bf Step 1}:  Let $I$ be a random index chosen uniformly from the set $\{1,2,\cdots, n\}$ and independent of all others.
	Noting that $(\xi, \xi^{\{i\}})$ is an exchangeable pair for $1\le i\le n$, we conclude that $(\xi, \xi^{\{I\}})$ is also exchangeable.
	Define
	\bestar
	D= {Y_I-Y_I^{\{I\}}}  ~~~~\mbox{and}~~~~\Delta= {f(\xi)-f(\xi^{\{I\}})}.
	\eestar
	Observing that
	\bestar
	f(\xi)-\mathbb{E}(f(\xi))=\sum_{j=1}^n Y_j~~~~\mbox{and}~~~~f(\xi^{\{i\}})-\mathbb{E}(f(\xi))=\sum_{j=i}^n Y_j^{\{i\}}+\sum_{j=1}^{i-1} Y_j,
	\eestar
	we can rewrite $\Delta$ as
	$\Delta=D+\tilde{\Delta}$ with
	\bestar
	\tilde{\Delta}= \sum_{I< j\le n} (Y_j-Y_j^{\{I\}}).
	\eestar
	By letting $\lambda=1/n$, it follows from Lemma  \ref{lemmaY1} that
	\be 
	\mathbb{E}(D|\mathscr{F}_n)=\frac{1}{n} \sum_{i=1}^n \mathbb{E}(Y_i-Y_i^{\{i\}}|\mathscr{F}_n)=\frac{1}{n} \sum_{i=1}^n Y_i=\lambda{W}.  \label{condD}
	\ee 

By using arguments  similar to those in the proof of Theorem 2.2 in \cite{S2019},
we obtain
\be
	d_K(W, \mathbf{Z}) \le \mathbb{E}\Big|1-\frac{1}{2\lambda}\mathbb{E}\big(D\Delta \big|\mathscr{F}_n\big) \Big|+\frac{1}{\lambda} \mathbb{E}\big|\mathbb{E}\big(|D|\Delta \big| \mathscr{F}_n\big)\big|.  \label{dKexpre}
	\ee

We now present a detailed proof of \eqref{dKexpre}.
For any given $z\in \mathbb{R}$, let $g:=g_z$
be the solution to the Stein equation
\be
g'(w)-wg(w)=I(w\le z)-\Phi(z).  \label{Steineq}
\ee
Define $W'=f(\xi^{\{I\}})-\mathbb{E} (f(\xi))$. By \eqref{condD},  we have
\bestar
0&=&\mathbb{E}(D(g(W)+g(W')))=2\mathbb{E}(Dg(W))-\mathbb{E}(D(g(W)-g(W')))\\
&=& 2\lambda \mathbb{E}(Wg(W))-\mathbb{E}\Big(D\int_{-\Delta}^0 g'(W+t)dt\Big),
\eestar
and consequently,
\bestar
\mathbb{E}(Wg(W))=\frac{1}{2\lambda}\mathbb{E}\Big(D\int_{-\Delta}^0 g'(W+t)dt\Big).
\eestar
Therefore,
\be 
&&\mathbb{P}(W\le z)-\Phi(z)=
\mathbb{E}(g'(W)-Wg(W))=I_0-I_1,  \label{BE01}
\ee 
where 
\bestar
I_0&=&\mathbb{E} \Big(g'(W) \Big(1-\frac{1}{2\lambda} \mathbb{E}(D\Delta|\mathscr{F}_n)\Big)\Big),\\
I_1&=&\frac{1}{2\lambda}\mathbb{E}\Big(D\int_{-\Delta}^0 (g'(W+t)-g'(W))dt\Big).
\eestar
By using the Stein equation \eqref{Steineq}, we have
\bestar
I_1&=&\frac{1}{2\lambda}\mathbb{E}\Big(D\int_{-\Delta}^0 ((W+t)g(W+t)-Wg(W))dt\Big)\\
&&~~~~+\frac{1}{2\lambda}\mathbb{E}\Big(D\int_{-\Delta}^0 (I(W+t\le z)-I(W\le z))dt\Big).
\eestar
Similar arguments to those in  the proof of Theorem 2.2 in \cite{S2019} yield that
\bestar
&&0\ge \int_{-\Delta}^0 ((W+t)g(W+t)-Wg(W))dt \ge -\Delta (Wg(W)- W'g(W')),\\
&&0\le \int_{-\Delta}^0 (I(W+t\le z)-I(W\le z))dt \le \Delta (I(W'\le z)-I(W\le z)),
\eestar
and
\bestar
|I_0|&\le& \mathbb{E} \Big|1-\frac{1}{2\lambda} \mathbb{E}(D\Delta|\mathscr{F}_n)\Big|,\\
I_1&\le& \frac{1}{2\lambda} \mathbb{E}(D^-\Delta (Wg(W)-W'g(W')))+\frac{1}{2\lambda} \mathbb{E}(D^+\Delta (I(W'\le z)-I(W\le z)))\\
&=&\frac{1}{2\lambda} \mathbb{E}(|D|\Delta Wg(W))-\frac{1}{2\lambda} \mathbb{E}(|D|\Delta I(W\le z))\le \frac{1}{\lambda} \mathbb{E}|\mathbb{E}(|D|\Delta|\mathscr{F}_n)|,
\eestar
where $D^+=\max\{D, 0\}, D^-=\max\{-D, 0\}$, and we have used the facts that
$ \mathbb{E}(D^-\Delta W'g(W'))=-\mathbb{E}(D^+\Delta W g(W))$
and
$\mathbb{E}(D^+\Delta (I(W'\le z))=- \mathbb{E}(D^-\Delta (I(W\le z))$.
Hence, by \eqref{BE01}, we have
\bestar
\mathbb{P}(W\le z)-\Phi(z)  \ge  -\mathbb{E}\Big|1-\frac{1}{2\lambda}\mathbb{E}(D\Delta|\mathscr{F}_n) \Big|-\frac{1}{\lambda} \mathbb{E}\big|\mathbb{E}(|D|\Delta| \mathscr{F}_n)\big|.
\eestar
Similarly, we can also get that
\bestar
\mathbb{P}(W\le z)-\Phi(z)  \le  \mathbb{E}\Big|1-\frac{1}{2\lambda}\mathbb{E}(D\Delta|\mathscr{F}_n) \Big|+\frac{1}{\lambda} \mathbb{E}\big|\mathbb{E}(|D|\Delta|\mathscr{F}_n)\big|.
\eestar
Therefore \eqref{dKexpre} follows immediately.

{\bf Step 2}:  In view of \eqref{dKexpre},  to prove  Theorem \ref{thBE1}, it suffices to show that	
	\be
	\frac{1}{\lambda}\mathbb{E} \big|\mathbb{E} \big(D\tilde{\Delta} \big|\mathscr{F}_n \big) \big|&\le& \Big(\sum_{j=1}^n \mathbb{E}\Big(\sum_{i=1}^{j-1}\Delta_i\Delta_{ij}\Big)^2\Big)^{1/2},
	\label{DDelta1}\\
	\frac{1}{\lambda}\mathbb{E} \big|\mathbb{E} \big(|D|\tilde{\Delta} \big|\mathscr{F}_n \big) \big|
	&\le& \Big(\sum_{j=1}^n \mathbb{E}\Big(\sum_{i=1}^{j-1}|\Delta_i|\Delta_{ij}\Big)^2\Big)^{1/2},  \label{DDelta2}\\
	\frac{1}{\lambda}\mathbb{E} \big|\mathbb{E} \big(|D|D  \big|\mathscr{F}_n \big) \big| &\le& 
	\Big(\sum_{i=1}^n \mathbb{E}(\Delta_i^4)\Big)^{1/2},\label{DDelta3} \\
	\mathbb{E}\Big|1-\frac{1}{2\lambda}\mathbb{E} \big(D^2 \big|\mathscr{F}_n \big)\Big| &\le&  \Big(\sum_{j=1}^n \mathbb{E}\Big(\sum_{i=j+1}^{n}\Delta_i\Delta_{ij}\Big)^2
	+4\sum_{i=1}^n \mathbb{E}(\Delta_i^4)\Big)^{1/2}.
	\label{DDelta4}
	\ee
Indeed,	by applying \eqref{DDelta1}$-$(\ref{DDelta4}), we have
	\bestar
	d_K(W, \mathbf{Z})&\le& \mathbb{E}\Big|1-\frac{1}{2\lambda}\mathbb{E} \big(D^2 \big|\mathscr{F}_n \big)\Big|+\frac{1}{2\lambda}\mathbb{E} \big|\mathbb{E} \big(D\tilde{\Delta} \big|\mathscr{F}_n \big) \big|+\frac{1}{\lambda}\mathbb{E} \big|\mathbb{E} \big(|D|D \big|\mathscr{F}_n \big) \big|+\frac{1}{\lambda}\mathbb{E} \big|\mathbb{E} \big(|D|\tilde{\Delta} \big|\mathscr{F}_n \big) \big|\\
	&\le& 5\Big(\sum_{j=1}^n\mathbb{E}\Big(\sum_{i\ne j} |\Delta_i||\Delta_{ij}|\Big)^2
	+\sum_{i=1}^n   \mathbb{E}(\Delta_i^4)\Big)^{1/2}.
	\eestar
This proves Theorem \ref{thBE1}.

We now provide the proof of \eqref{DDelta1}$-$(\ref{DDelta4}).  
	
By using the basic properties of conditional expectations, we have 
	\be 
	\frac{1}{\lambda}\mathbb{E}|\mathbb{E}(D\tilde{\Delta}|\mathscr{F}_n)|
	&\le&\frac{1}{\lambda}(\mathbb{E}(\mathbb{E}^2(D\tilde{\Delta}|\mathscr{F}_n)))^{1/2}
	=\frac{1}{\lambda}\big(\mathbb{E}\big(\mathbb{E}^2(\mathbb{E}(D\tilde{\Delta}|\mathscr{F}_n, \mathscr{F}_n')|\mathscr{F}_n)\big)\big)^{1/2}
	\nonumber\\
	&\le& \frac{1}{\lambda}\big(\mathbb{E}\big(\mathbb{E}(\mathbb{E}^2(D\tilde{\Delta}|
	\mathscr{F}_n, \mathscr{F}_n')|\mathscr{F}_n)\big)\big)^{1/2}\nonumber\\
	&=&\frac{1}{\lambda}\big(\mathbb{E}\big(\mathbb{E}^2(D\tilde{\Delta}|
	\mathscr{F}_n, \mathscr{F}_n')\big)\big)^{1/2}=\Big(\mathbb{E}\Big(\sum_{j=1}^n \sum_{i=1}^{j-1}  D_{1ij}\Big)^2\Big)^{1/2},\label{EDDelta}
	\ee 
where $D_{1ij}$ is defined in \eqref{D1ijD2ij}.
For $i<j$, it follows from   Lemma  \ref{lemmaY1} that 
\bestar
 D_{1ij}=\Delta_i \mathbb{E}(\Delta_{ij} |\mathscr{F}_{j-1}',\mathscr{F}_j)
= \mathbb{E}(\Delta_i \Delta_{ij} |\mathscr{F}_{j-1}',\mathscr{F}_j).
\eestar
Therefore,  by Lemma   \ref{lemmaY2},  we have
	\bestar
	\mathbb{E}\Big(\sum_{j=1}^n \sum_{i=1}^{j-1}  D_{1ij}\Big)^2
&=&\sum_{j=1}^n \mathbb{E}\Big(\sum_{i=1}^{j-1}   D_{1ij}\Big)^2=\sum_{j=1}^n \mathbb{E}\Big(\mathbb{E}\Big(\sum_{i=1}^{j-1}\Delta_i\Delta_{ij}\Big|\mathscr{F}_{j-1}',\mathscr{F}_j\Big)\Big)^2\\
&\le& \sum_{j=1}^n \mathbb{E}\Big(\sum_{i=1}^{j-1}\Delta_i\Delta_{ij}\Big)^2.
	\eestar
	This, together with (\ref{EDDelta}),  implies (\ref{DDelta1}).

	Similarly, we can obtain that
	\bestar
	\frac{1}{\lambda}\mathbb{E}|\mathbb{E}(|D|\tilde{\Delta}|\mathscr{F}_n)|\le\Big(\sum_{j=1}^n \mathbb{E}\Big(\sum_{i=1}^{j-1}  D_{2ij}\Big)^2\Big)^{1/2}  \le \Big(\sum_{j=1}^n \mathbb{E}\Big(\sum_{i=1}^{j-1}|\Delta_i|\Delta_{ij}\Big)^2\Big)^{1/2},
	\eestar
	and
	\bestar
	\frac{1}{\lambda}\mathbb{E}|\mathbb{E}(|D|D|\mathscr{F}_n)|\le\Big(\sum_{i=1}^n \mathbb{E} ( D_{2ii}^2)\Big)^{1/2}\le
	\Big(\sum_{i=1}^n \mathbb{E}(\Delta_i^4)\Big)^{1/2},
	\eestar
where $\{D_{2ij}, 1\le i\le j\le n\}$ is defined in \eqref{D1ijD2ij}.
 Therefore, (\ref{DDelta2}) and (\ref{DDelta3})  hold.

Finally, we prove (\ref{DDelta4}).
It follows from Lemma  \ref{lemmaY1} that
\bestar
\mathbb{E}(D^2)=\mathbb{E}(\mathbb{E}(D^2|\mathscr{F}_n, \mathscr{F}_{n}'))
=\frac{1}{n}\sum_{i=1}^n \mathbb{E}(Y_i-Y_i^{\{i\}})^2=\frac{2}{n}=2\lambda.
\eestar	
Hence, by using   similar arguments to those in (\ref{EDDelta}), we have
	\be
	\mathbb{E}\Big|1-\frac{1}{2\lambda}\mathbb{E}(D^2|\mathscr{F}_n)\Big|&\le& \frac{1}{2\lambda}(\mbox{Var}(\mathbb{E}(D^2|\mathscr{F}_n)))^{1/2}
	\le \frac{1}{2\lambda}(\mbox{Var}(\mathbb{E}(D^2|\mathscr{F}_n, \mathscr{F}_n')))^{1/2}\nonumber\\
	&\le& \frac{1}{2}\Big(\mbox{Var}\Big(\sum_{i=1}^n   (Y_i-Y_i^{\{i\}})^2\Big)\Big)^{1/2}.  \label{1-d21}
	\ee	
Note that 
 \be
\Delta_i=Y_i-Y_i^{\{i\}}=g_i(\xi_1, \cdots, \xi_i)-g_i(\xi_1,\cdots, \xi_{i-1}, \xi_i')
:=\tilde{g}_i(\xi_1, \cdots, \xi_i, \xi_i'),  \label{tildegi}
\ee where $\tilde{g}_i: \mathbb{R}^{i+1}\rightarrow \mathbb{R}$ is a measurable function.
Let $(\xi_1^*, \xi_1'^*, \cdots, \xi_n^*, \xi_n'^*)$ be an independent copy of $(\xi_1,\xi_1',\cdots, \xi_n, \xi_n')$, and define
	\bestar
	V_{ij}^*=\left\{
	\begin{array}{ll}
		\tilde{g}_i(\xi_1,  \cdots, \xi_{j-1}, \xi_j^*, \xi_{j+1},\cdots, \xi_i, \xi_i'), & j<i,\\
		\tilde{g}_i(\xi_1, \cdots, \xi_{j-1}, \xi_j^*, \xi_j'^{*}), & j=i.
	\end{array}
	\right.
	\eestar
	By the Efron-Stein inequality (see  \cite{ES1981}), we have
	\be
	\mbox{Var}\Big(\sum_{i=1}^n   (Y_i-Y_i^{\{i\}})^2\Big)&\le&\frac{1}{2}\sum_{j=1}^n \mathbb{E}\Big(\sum_{i=j}^n (\Delta_i^2-(V_{ij}^*)^2) \Big)^2\nonumber\\
	&=& \frac{1}{2}\sum_{j=1}^n \mathbb{E}\Big(\sum_{i=j}^n (\Delta_i-V_{ij}^*)(\Delta_i+V_{ij}^*) \Big)^2\nonumber\\
	&\le& \sum_{j=1}^n \mathbb{E}\Big(\sum_{i=j}^n (\Delta_i-V_{ij}^*)\Delta_i \Big)^2+\sum_{j=1}^n \mathbb{E}\Big(\sum_{i=j}^n (\Delta_i-V_{ij}^*)V_{ij}^* \Big)^2\nonumber\\
	&=& 2\sum_{j=1}^n \mathbb{E}\Big(\sum_{i=j}^n (\Delta_i-V_{ij}^*)\Delta_i \Big)^2.   \label{1-d22}
	\ee
	Observe that for fixed $j$, we have $V_{jj}^*\stackrel{d}{=}\Delta_j$ and
	\bestar
	(\Delta_i, V_{ij}^*, i=j+1,\cdots, n)\stackrel{d}{=} (\Delta_i, V_{ij}, i=j+1,\cdots, n),
	\eestar
	where (by (\ref{A4.20}) and (\ref{tildegi}))
	\bestar
	V_{ij}&=&\tilde{g}_i(\xi_1,  \cdots, \xi_j',\cdots, \xi_i, \xi_i')=\mathbb{E}(f(\xi^{\{j\}})-f(\xi^{\{i,j\}})|\xi_i', \mathscr{F}_i^{\{j\}})\\
	&=&\mathbb{E}(f(\xi^{\{j\}})-f(\xi^{\{i,j\}})|\xi_i', \xi_j', \mathscr{F}_i), ~~~~i>j,
	\eestar
and  $\mathscr{F}_i^{\{j\}}=\sigma(\xi_1,\cdots, \xi_{i-1}, \xi_j',  \xi_{i+1},\cdots, \xi_{i})$.
	Since for $i>j$, $$\Delta_i=\mathbb{E}(f(\xi)-f(\xi^{\{i\}})|\xi_i',\mathscr{F}_i)=\mathbb{E}(f(\xi)-f(\xi^{\{i\}})|\xi_i',\xi_j', \mathscr{F}_i),$$ we have
	\bestar
	\Delta_i-V_{ij}=\mathbb{E}(f(\xi)-f(\xi^{\{i\}})-f(\xi^{\{j\}})+f(\xi^{\{i,j\}})|\xi_i',\xi_j', \mathscr{F}_i)=\Delta_{ij}, ~~~~i>j.
	\eestar
Hence,
	\bestar
	\mathbb{E}\Big(\sum_{i=j}^n (\Delta_i-V_{ij}^*)\Delta_i \Big)^2
	&\le& 2\mathbb{E}\Big(\sum_{i=j+1}^n (\Delta_i-V_{ij}^*)\Delta_i \Big)^2+2\mathbb{E}( (\Delta_j-V_{jj}^*)^2\Delta_j^2 )\\
	&\le & 2\mathbb{E}\Big(\sum_{i=j+1}^n (\Delta_i-V_{ij})\Delta_i \Big)^2+4\mathbb{E}(\Delta_j^4)+4\mathbb{E}( (V_{jj}^*)^2\Delta_j^2 )\\
	&\le& 2\mathbb{E}\Big(\sum_{i=j+1}^{n}\Delta_i\Delta_{ij}\Big)^2+8\mathbb{E}(\Delta_j^4),
	\eestar
where we have used the inequality $\mathbb{E}( (V_{jj}^*)^2\Delta_j^2 )
\le \sqrt{\mathbb{E}( (V_{jj}^*)^4 ) \mathbb{E}(\Delta_j^4)}= \mathbb{E}(\Delta_j^4)$, since $V_{jj}^*\stackrel{d}{=}\Delta_j$.
This, together with (\ref{1-d21}) and (\ref{1-d22}),  implies (\ref{DDelta4})
and completes  the proof of Theorem \ref{thBE1}.
\end{proof}

\subsection{Proof of Proposition \ref{proppercolation}}

\begin{proof} 
We begin by expressing $N_{n,d}$ as a function of a sequence of certain independent random variables, based on the fact that for any $k\in \mathbb{Z}$,
\be
\frac{1}{2\pi}\int_{-\pi}^{\pi}  e^{\textbf{i}tk}dt=\left\{
\begin{array}{cc}
1, & \mbox{if}~ k=0,\\
0, & \mbox{if}~ k\ne 0.
\end{array}
\right. \label{inte}
\ee

Let $\widetilde{D}_{n,i}$ denote the degree of vertex $i$ in the  percolated subgraph.
Define
\bestar
\xi_{j}=I(e_j\mbox{~is kept in the  percolated subgraph}),~~~~j=1,2,\cdots, m.
\eestar
Then $\xi_1, \xi_2, \cdots, \xi_{m}$ are i.i.d.  with
$
\mathbb{P}(\xi_{j}=1)=1-\mathbb{P}(\xi_{j}=0)=\widetilde{p}\,
$,
and
$
\widetilde{D}_{n,i}=\sum_{j\in A_i} \xi_{j}
$, where
\bestar
A_i=\Big\{j\in \{1,2,\cdots, m\}: ~i \mbox{~is one of the endpoints of}~ e_j\Big\}.
\eestar
Let $\xi=(\xi_1,\cdots, \xi_m),~\textbf{i}=\sqrt{-1}, $ and let $f_d: \{0,1\}^{m} \rightarrow \mathbb{R}$ be a measurable function such that
\bestar
f_d(x)=\frac{1}{2\pi}\sum_{i=1}^n\int_{-\pi}^{\pi} e^{-\textbf{i}td} \prod_{j\in A_i} e^{\textbf{i}tx_{j}}dt,~~~~x=(x_1,\cdots, x_m) \in \{0,1\}^{m}.
\eestar
Then  by (\ref{inte}), we have
\bestar
N_{n,d}&=&\sum_{i=1}^n I(\widetilde{D}_{n,i}=d)=\frac{1}{2\pi}\sum_{i=1}^n\int_{-\pi}^{\pi}  e^{\textbf{i}t(\widetilde{D}_{n,i}-d)}dt
\nonumber\\
&=&\frac{1}{2\pi}\sum_{i=1}^n\int_{-\pi}^{\pi} e^{-\textbf{i}td} \prod_{j\in A_i}  e^{\textbf{i}t\xi_{j}}dt
=f_d(\xi).   
\eestar

We now apply Theorem \ref{thBE1} to establish the desired conclusion.
Let  $\xi'=(\xi_1', \cdots, \xi_m')$ be an independent copy of $\xi$ and define $\xi^{A}$ as in (\ref{XA}) for any $A\subseteq \{1,2, \cdots, m\}$.
Observe that for any $1\le i\le m$,
\bestar
f_d(\xi)-f_d(\xi^{\{i\}})=\frac{1}{2\pi}\int_{-\pi}^{\pi} e^{-\textbf{i}td}(e^{\textbf{i}t\xi_i}-e^{\textbf{i}t\xi_i'}) \Big( \prod_{j\in A_{e_i^+},~j\ne i} e^{\textbf{i}t\xi_{j}}
+ \prod_{j\in A_{e_i^-},~j\ne i} e^{\textbf{i}t\xi_{j}}\Big)dt,
\eestar
where $e_i^+, e_i^-\in V_n$ are the two  endpoints of $e_i$.
Since $\xi_i, \xi_i'\in \{0,1\}$, we have
\bestar
|f_d(\xi)-f_d(\xi^{\{i\}})|\le \frac{1}{\pi}\int_{-\pi}^{\pi} |e^{\textbf{i}t\xi_i}-e^{\textbf{i}t\xi_i'}|dt=\frac{1}{\pi}\int_{-\pi}^{\pi} |(e^{\textbf{i}t}-1)(\xi_i-\xi_i')|dt\le 3|\xi_i-\xi_i'|.
\eestar
Moreover, if $e_i$ and $e_j$  do not share a common endpoint, then $$f_d(\xi)-f_d(\xi^{\{i\}})-f_d(\xi^{\{j\}})+f_d(\xi^{\{i,j\}})=0.$$
If $i\ne j$ and $e_i$ and $e_j$  share a common endpoint (denoted as $e_i\sim e_j$), then
\bestar
|f_d(\xi)-f_d(\xi^{\{i\}})-f_d(\xi^{\{j\}})+f_d(\xi^{\{i,j\}})|&\le &\frac{1}{2\pi}\int_{-\pi}^{\pi} |(e^{\textbf{i}t\xi_i}-e^{\textbf{i}t\xi'_i})
 (e^{\textbf{i}t\xi_j}-e^{\textbf{i}t\xi_j'})|dt\\
&=&\frac{1}{2\pi}\int_{-\pi}^{\pi}|(e^{\textbf{i}t}-1)^2 (\xi_i-\xi'_i)
 (\xi_j-\xi'_j)|dt\\
 &=& 2|\xi_i-\xi'_i||\xi_j-\xi'_j|.
\eestar
Define $\Delta_i$ and $\Delta_{i,j}$ as in Theorem \ref{thBE1}. Then $|\Delta_i|\le 3|\xi_i-\xi_i'|$ and $|\Delta_{i,j}|\le 2|\xi_i-\xi'_i||\xi_j-\xi'_j|I(e_i\sim e_j)$ for $i\ne j$.
Hence,
\be
\sum_{i=1}^{m}\mathbb{E}(\Delta_i^4)\le 81\sum_{i=1}^{m}\mathbb{E}((\xi_i-\xi'_i)^4)=81\sum_{i=1}^{m}\mathbb{E}(|\xi_i-\xi'_i|)= 162\, m \widetilde{p}\,
(1-\widetilde{p}\,),  \label{Deltai1}
\ee
and
\bestar
\sum_{j=1}^{m}\mathbb{E}\Big(\sum_{i\ne j} |\Delta_i||\Delta_{ij}|\Big)^2&\le& 36\sum_{j=1}^{m}\mathbb{E}\Big(\sum_{i\ne j} |\xi_i-\xi'_i||\xi_j-\xi'_j|I(e_i\sim e_j)\Big)^2\\
&=&144\, \widetilde{p}\,^2(1-\widetilde{p}\,)^2\sum_{j=1}^{m} \sum_{i\ne j}I(e_i\sim e_j) \\
&&~~~~+ 288 \,\widetilde{p}\,^3(1-\widetilde{p}\,)^3\sum_{j=1}^{m} \sum_{i_1\ne j}\sum_{i_2\not\in \{i_1,j\}}I(e_{i_1}\sim e_j)I(e_{i_2}\sim e_j).
\eestar
By observing that
$
\sum_{i\ne j}I(e_i\sim e_j)=d_{n, e_j^+}+d_{n, e_j^-}-2
$
for any fixed $j$,
we have
\bestar
\sum_{j=1}^{m}\sum_{i\ne j}I(e_i\sim e_j)\le \sum_{j=1}^{m} (d_{n, e_j^+}+d_{n, e_j^-})=\sum_{i=1}^n d_{n,i}^2,
\eestar
and
\bestar
&&\sum_{j=1}^{m} \sum_{i_1\ne j}\sum_{i_2\not\in \{i_1,j\}}I(e_{i_1}\sim e_j)I(e_{i_2}\sim e_j)\\
&\le& \sum_{j=1}^{m} \Big(\sum_{i\ne j}I(e_i\sim e_j)\Big)^2=\sum_{j=1}^{m} (d_{n, e_j^+}+d_{n, e_j^-}-2)^2\\
&\le& 2\sum_{j=1}^{m} (d_{n, e_j^+}^2+d_{n, e_j^-}^2)=2\sum_{i=1}^{n} d_{n,i}^3.
\eestar
Therefore,
\be
\sum_{j=1}^m\mathbb{E}\Big(\sum_{i\ne j} |\Delta_i||\Delta_{ij}|\Big)^2\le 144\, \widetilde{p}\,^2(1-\widetilde{p}\,)^2 \sum_{i=1}^n d_{n,i}^2+576 \,\widetilde{p}\,^3(1-\widetilde{p}\,)^3 \sum_{i=1}^n d_{n,i}^3.  \label{Deltai2}
\ee
Note that
\bestar
 \widetilde{p}\,^2(1-\widetilde{p}\,)^2\sum_{i=1}^n d_{n,i}^2&\le&\widetilde{p}\,^2(1-\widetilde{p}\,)^2 \Big(\sum_{i=1}^n d_{ni}\sum_{i=1}^n d_{n,i}^3\Big)^{1/2}=\widetilde{p}\,^2(1-\widetilde{p}\,)^2 \Big(2m\sum_{i=1}^n d_{n,i}^3\Big)^{1/2}\\
 &\le& m\widetilde{p}\,(1-\widetilde{p}\,)+ \widetilde{p}\,^3(1-\widetilde{p}\,)^3 \sum_{i=1}^n d_{n,i}^3.
 \eestar
 The desired result follows from (\ref{corthBE1}),  (\ref{Deltai1}) and (\ref{Deltai2}).
\end{proof}

\subsection{Proof of Proposition \ref{propa41}}

\begin{proof}
Recall that
$\mu(\mathbb{T}_n)=\sum_{i=1}^n p^{D_{n,i}}$ and 
\be
D_{n,i}=I(i\ne 1)+\sum_{j=i+1}^n I_{i,j}, \label{Dnir2}
\ee
 where
 $I_{i,j}$ is defined in (\ref{Iij}). Note that in (\ref{Dnir2}), the expression for $D_{n,i}$ differs between the case $i=1$ and $i \ne 1$. 
To simplify the analysis in the subsequent proof, we replace the total degree $D_{n,i}$ with the outdegree $$D_{n,i}^* =\sum_{j=i+1}^n I_{i,j},$$
 and first establish a Berry–Esseen bound for the corresponding quantity $\mu^*(\mathbb{T}_n):=\sum_{i=1}^{n-1} p^{D_{n,i}^*}$.

On the other hand, since $D_{n,i}=I(i \ne 1)+ D_{n,i}^*$, it follows that
\be 
\mu(\mathbb{T}_n)=p\, \mu^*(\mathbb{T}_n)+(1-p)p^{D_{n,1}^*}+p,  \label{mudiff}
\ee 
where $(1-p)p^{D_{n,1}^*}+p\in (0,1]$. Therefore, the  result for $\mu(\mathbb{T}_n)$ can be easily derived from the corresponding one for $\mu^*(\mathbb{T}_n)$.

First, we  apply Theorem \ref{thBE1} to establish a Berry-Esseen bound for $\mu^*(\mathbb{T}_n)$.
Let
\bestar
\sigma_*^{2}=\mbox{Var}(\mu^*(\mathbb{T}_n))~~~~~\mbox{and}~~~~~  W_*=\frac{\mu^*(\mathbb{T}_n)-\mathbb{E}(\mu^*(\mathbb{T}_n))}{\sigma_*}.
\eestar
Define the measurable function $f_*:\{1,2,\cdots, n-1\}^{n-1}\rightarrow \mathbb{R}$ by
\bestar
f_*(x)=\sum_{i=1}^{n-1} p^{\sum_{j=i+1}^n I(x_j=i)},~~~~x=(x_2,\cdots, x_{n}).
\eestar
Let $U=(U_2,\cdots, U_n)$, and let $U'=(U_2', \cdots, U_n')$ be an independent copy of $U$.  Define $U^{(A)}$ as in (\ref{XA}) for any $A\subseteq \{2,3,\cdots, n\}$.
Then 
\bestar
\sum_{i=1}^{n-1} p^{D_{n,i}^*}=\sum_{i=1}^{n-1} p^{\sum_{j=i+1}^n I(U_j=i)}=f_*(U)
\eestar
and
\bestar
f_*(U)-f_*(U^{\{i\}})=\sum_{k=1}^{i-1} (p^{ I_{k,i} }-p^{ I_{k,i}'})\prod_{k<j\le n, j\ne i}p^{I_{k,j}},
\eestar
where $I'_{k,i}=I(U_{i}'=k)$.
Define $\Delta_i$ and $\Delta_{i,j}$ as in Theorem \ref{thBE1}.
We have
\be
|\Delta_i|\le \sum_{k=1}^{i-1} |p^{ I_{k,i} }-p^{ I_{k,i}' }|=(1-p) \sum_{k=1}^{i-1} |I_{k,i}-I_{k,i}'|\le \sum_{k=1}^{i-1} Y_{k,i},   \label{Deltaia}
\ee
where $Y_{k,i}=|I_{k,i}-I_{k,i}'|$.
Furthermore, for $i\ne j$,
\bestar
|\Delta_{ij}|
\le \sum_{k=1}^{i\wedge j-1} |p^{I_{k,i}}-p^{I_{k,i}'}| |p^{I_{k,j}}-p^{I_{k,j}'}|
\le  \sum_{k=1}^{i\wedge j-1} Y_{k,i}Y_{k,j}.
\eestar
 Therefore,
\be 
\sum_{j=2}^n \mathbb{E}\Big(\sum_{i\ne j} |\Delta_i||\Delta_{ij}|\Big)^2\le  \sum_{\mathbf{I}\in \mathcal{P}} \mathbb{E}(Y_{k_1, i}Y_{k_2, i}Y_{k_2,j}Y_{k_1',i'}Y_{k_2', i'}Y_{k_2', j}), \label{Diij}
\ee 
where $\mathbf{I}:=(i, j, k_1,k_2, i', k_1', k_2')$ and
  $$ \mathcal{P}:=\{\mathbf{I}: 1\le k_1,k_2<i\le n, ~1\le k_1', k_2' <i'\le n, ~ k_2,k_2'<j\le n,~i,i'\ne j \}.$$
Similarly,
\bestar
\sum_{i=2}^n \mathbb{E} (\Delta_i^4) \le  \sum_{i=2}^n \mathbb{E} \Big(\sum_{k=1}^{i-1} Y_{k,i}\Big)^4=\sum_{i=2}^n \sum_{k_1,k_2,k_1',k_2'=1}^{i-1} \mathbb{E}(Y_{k_1, i}Y_{k_2, i}Y_{k_1',i}Y_{k_2', i}).
\eestar
Elementary calculations show that
\be
\sum_{j=2}^n \mathbb{E}\Big(\sum_{i\ne j} |\Delta_i||\Delta_{ij}|\Big)^2+\sum_{i=2}^n \mathbb{E} (\Delta_i^4) \le Cn. \label{BoundDij}
\ee
The detailed proof of (\ref{BoundDij}) is provided in \ref{sectappendix2}.
It follows from (\ref{Varsump*}) that $\sigma_*^{2}=p^{-2}\sigma_{1}^2n+O(1)$. By (\ref{corthBE1}) and  (\ref{BoundDij}), we have
\be
d_K(W_*, \mathbf{Z})\le Cn^{-1/2}.  \label{BEWstar}
\ee

We now turn to the proof of (\ref{BE2.2}).
Define
\bestar
\rho_n=\frac{p\sigma_*}{\sigma_{1}\sqrt{n}}, ~~~~R_n=\frac{1}{\sigma_{1}\sqrt{n}} \Big(\mu(\mathbb{T}_n)-p\mu^*(\mathbb{T}_n)+p\Big( \mathbb{E}(\mu^*(\mathbb{T}_n))-\frac{n}{2-p}\Big)\Big).
\eestar
Then we have
\bestar
\frac{\mu(\mathbb{T}_n)-\frac{np}{2-p}}{\sigma_{1}\sqrt{n}}=\rho_n W_* + R_n.
\eestar
By (\ref{mudiff}), (\ref{Esump*}) and (\ref{Varsump*}), we obtain that $\rho_n=1+O(n^{-1})$ and $|R_n|\le c_2n^{-1/2}$ for some $c_2>0$.
Hence
\be
d_K\Big(\frac{\mu(\mathbb{T}_n)-\frac{np}{2-p}}{\sigma_{1}\sqrt{n}}, \mathbf{Z}\Big)
\le \max\{ d_n(c_2), ~d_n(-c_2)\},   \label{BEWstar2}
\ee
where
\bestar
d_n(c)=\sup_{x\in \mathbb{R}}\Big|\mathbb{P}(W_*\le \rho_n^{-1}(x-c n^{-1/2}))-\Phi(x)\Big|,~~~~c\in \mathbb{R}.
\eestar
Applying Lemma 5.2 in \cite{Petrov1995} gives $\sup_{x\in \mathbb{R}} |\Phi(\rho_n^{-1}(x-c n^{-1/2}))-\Phi(x)|\le Cn^{-1/2}$. By (\ref{BEWstar}), we have
\bestar
d_n(c)\le d_K(W_*, \mathbf{Z})
+\sup_{x\in \mathbb{R}} |\Phi(\rho_n^{-1}(x-c n^{-1/2}))-\Phi(x)|\le Cn^{-1/2}.
\eestar
Therefore  the desired result follows from (\ref{BEWstar2}).
\end{proof}

\subsection{Proof of Proposition \ref{lemmanu1nBE}}

Applying Proposition \ref{proppercolation} to $\mathbb{T}_n$ with $\widetilde{p}=1-p$ gives
\be
\sup_{x\in \mathbb{R}}\Big|\mathbb{P}\Big(\frac{\nu_1(n)-\mu(\mathbb{T}_n)}{\sigma(\mathbb{T}_n)}\le x \Big| \mathbb{T}_n \Big)-\Phi(x)\Big|\le \frac{C\sqrt{n}+C(\sum_{i=1}^n D_{n,i}^3)^{1/2}}{\sigma^2(\mathbb{T}_n)},  \label{BEmu}
\ee
where $D_{n,i}, \mu(\mathbb{T}_n)$ and $\sigma^2(\mathbb{T}_n)$ are defined in (\ref{Dni}), (\ref{munT}) and (\ref{sigmanT}), respectively.
 Lemma \ref{lemmaadd1} shows the existence of $c_3>0$ such that
$
\mathbb{E}(\sigma^2(\mathbb{T}_n))\ge c_3n
$
holds for all $n\ge 1$, and hence
\bestar
\mathbb{P}(\widetilde{E}_n^c)\le \mathbb{P}(\sigma^2(\mathbb{T}_n)-\mathbb{E}(\sigma^2(\mathbb{T}_n))\le -(c_3/2)n) \le \frac{\mbox{Var}(\sigma^2(\mathbb{T}_n))}{c_3^2n^2/4}\le Cn^{-1},
\eestar
where
$
\widetilde{E}_n=\{\sigma^2(\mathbb{T}_n)\ge (c_3/2)n\}.
$
It follows from Lemma \ref{lemma1} and (\ref{BEmu}) that
\bestar
 &&\mathbb{E}\Big(\sup_{x\in \mathbb{R}}\Big|\mathbb{P}\Big(\frac{\nu_1(n)-\mu(\mathbb{T}_n)}{\sigma(\mathbb{T}_n)}\le x \Big| \mathbb{T}_n \Big)-\Phi(x)\Big|\Big)\\
&\le& \mathbb{P}(\widetilde{E}_n^c)+\mathbb{E}\Big(\sup_{x\in \mathbb{R}}\Big|\mathbb{P}\Big(\frac{\nu_1(n)-\mu(\mathbb{T}_n)}{\sigma(\mathbb{T}_n)}\le x \Big| \mathbb{T}_n \Big)-\Phi(x)\Big|I_{\widetilde{E}_n}\Big)\\
&\le& Cn^{-1}+Cn^{-1} \Big(\sqrt{n}+\Big(\mathbb{E} \sum_{i=1}^n D_{n,i}^3\Big)^{1/2}\Big)\le Cn^{-1/2}.
\eestar
Hence
\be
&&\sup_{x\in \mathbb{R}}\Big|\mathbb{P}\Big(\frac{\nu_1(n)-\frac{np}{2-p}}{\sigma_{3}\sqrt{n}}\le x \Big)-\mathbb{E}\Big(\Phi\Big( \frac{\sigma_{3} x\sqrt{n}-\mu(\mathbb{T}_n)+\frac{np}{2-p}}{\sigma(\mathbb{T}_n)} \Big)\Big)\Big|\nonumber\\
&\le&  \mathbb{E}\Big(\sup_{x\in \mathbb{R}}\Big|\mathbb{P}\Big(\frac{\nu_1(n)-\mu(\mathbb{T}_n)}{\sigma(\mathbb{T}_n)}\le \frac{\sigma_{3} x\sqrt{n}-(\mu(\mathbb{T}_n)-\frac{np}{2-p})}{\sigma(\mathbb{T}_n)}\Big| \mathbb{T}_n \Big)\nonumber\\
&&~~~~~~~~~~~~~~~~~~~~~~~~~~~~~~~~~~~~~~-\Phi\Big( \frac{\sigma_{3} x\sqrt{n}-\mu(\mathbb{T}_n)+\frac{np}{2-p}}{\sigma(\mathbb{T}_n)} \Big)\Big|\Big)\le Cn^{-1/2},  \qquad \label{pro2.3.1}
\ee
where $\sigma_3^2$ is defined in (\ref{sigma3}).
Using similar arguments to those in the proofs of (\ref{diffadd2}) and (\ref{diffadd3}), together with Lemma \ref{lemmaadd1} and Proposition \ref{propa41}, gives
\be
&&\sup_{x\in \mathbb{R}}\Big|\mathbb{E}\Big(\Phi\Big( \frac{\sigma_{3} x\sqrt{n}-\mu(\mathbb{T}_n)+\frac{np}{2-p}}{\sigma(\mathbb{T}_n)} \Big)\Big)-
\mathbb{E}\Big(\Phi\Big( \frac{\sigma_{3} x\sqrt{n}-\mu(\mathbb{T}_n)+\frac{np}{2-p}}{\sigma_{2}\sqrt{n}} \Big)\Big)\Big|\nonumber\\
&\le& C\frac{\mathbb{E}|\sigma^2(\mathbb{T}_n)-\sigma_{2}^2n|}{\sigma_{2}^2n}
\le  C\frac{\mathbb{E}|\sigma^2(\mathbb{T}_n)-\mathbb{E}(\sigma^2(\mathbb{T}_n))|+|\mathbb{E}(\sigma^2(\mathbb{T}_n))-\sigma_{2}^2n|}{\sigma_{2}^2n}\nonumber\\
&\le&  C\frac{\sqrt{\mbox{Var}(\sigma^2(\mathbb{T}_n))}+|\mathbb{E}(\sigma^2(\mathbb{T}_n))-\sigma_{2}^2n|}{\sigma_{2}^2n}\le Cn^{-1/2},
 \label{pro2.3.2}
\ee
and
\be
&& \sup_{x\in \mathbb{R}}\Big|
\mathbb{E}\Big(\Phi\Big( \frac{\sigma_{3} x\sqrt{n}-\mu(\mathbb{T}_n)+\frac{np}{2-p}}{\sigma_{2}\sqrt{n}} \Big)\Big)-
\mathbb{E}\Big(\Phi\Big(\frac{\sigma_{3}  x-\mathbf{Z}_1}{\sigma_{2}}\Big)\Big)\Big|\le Cn^{-1/2},    \label{pro2.3.3}
\ee
where $\mathbf{Z}_1\sim N(0, \sigma_{1}^2)$,    $\sigma_{1}^2$ and $\sigma_{2}^2$ are defined in (\ref{sigma1}) and  (\ref{sigma2}) respectively.
By noting that $\sigma_{1}^2+\sigma_{2}^2=\sigma_{3}^2$, we obtain that
\bestar
\mathbb{E}\Big(\Phi\Big(\frac{\sigma_{3}  x- \mathbf{Z}_1}{\sigma_{2}}\Big)\Big)&=&\mathbb{P}\Big(\mathbf{Z}\le \frac{\sigma_{3}  x- \mathbf{Z}_1}{\sigma_{2}}\Big)
=\mathbb{P}\Big(\frac{\sigma_{2}\mathbf{Z}+\mathbf{Z}_1}{\sigma_{3} }\le x\Big)\\
&=& \mathbb{P}\Big(\frac{\sqrt{\sigma_{2}^2+ \sigma_{1}^2}}{\sigma_{3} }\mathbf{Z}\le x\Big)=\Phi(x),
\eestar
where $\mathbf{Z}$ is a standard normal random variable  and  independent of $\mathbf{Z}_1$.  This, together with  (\ref{pro2.3.1})-(\ref{pro2.3.3}), 
proves (\ref{nu1BE}) and completes the proof of Proposition \ref{lemmanu1nBE}.

\section*{Acknowledgment}
The author wishes to thank Professor Qi-Man Shao for some helpful  suggestions  and for providing the unpublished manuscript \cite{S2021}, in which the conclusion \eqref{dKexpre} is presented without proof and holds for a general $D$ that satisfies certain conditions. The author also thanks the referee for carefully reading the manuscript and providing many valuable suggestions. This work is supported by NSFC (No.11671373) and Hebei Natural Science Foundation (Grant no.\,A2025501005).

\newpage

\begin{appendices}
\renewcommand{\thesection}{Appendix A}
\renewcommand{\thesubsection}{A.\arabic{subsection}}
\renewcommand{\thelemma}{\Alph{section}.\arabic{lemma}}
\renewcommand{\theequation}{\Alph{section}.\arabic{equation}}

\section{Proofs of Lemmas \ref{lemma1}$-$\ref{lemmaadd1}} \label{secta3}

We present the proofs of Lemmas \ref{lemma1},   \ref{lemma38},   \ref{lemma43}, and  \ref{lemmaadd1}  in sequence. The proof of Lemma  \ref{lemmaadd2},
which depends on Lemmas  \ref{lemma43} and  \ref{lemmaadd1},  is placed last.

\subsection{Proof of Lemma \ref{lemma1}}
\begin{proof}
Recall that $I_{i,j}=I(U_j=i)$ for any $1\le i, j\le n$, and
\bestar
D_{n,1}=\sum_{j=2}^{n} I_{1,j},
~~~~
 D_{n,i}=1+\sum_{j=i+1}^{n}I_{i,j}, ~~~~2\le i\le n.
\eestar
We have
\bestar
\mathbb{E} (D_{n,i}-1)&=&\sum_{j=i+1}^n \mathbb{E}(I_{i,j})=\sum_{j=i+1}^n \frac{1}{j-1}\\
&\le& \int_i^n (x-1)^{-1}dx=\log (n-1)-\log (i-1),~~~~i\ge 2,
\eestar
and
\bestar
\mathbb{E} (D_{n,1})=\sum_{j=2}^n \mathbb{E}(I_{j,1})=\sum_{j=1}^{n-1} \frac{1}{j}\le 1+\log(n-1).
\eestar
Noting that $I_{i,j}, j=i+1,\cdots, n$ are independent for any fixed $i$ gives that
\bestar
\mathbb{E} (D_{n,i}-1)^l &=&\mathbb{E} \Big(\sum_{j=i+1}^n I_{i,j}\Big)^l=\sum_{j_1=i+1}^n \cdots\sum_{j_l=i+1}^n \mathbb{E} \Big(\prod_{k=1}^l I_{i, j_k}\Big)\\
&\le &  C\Big(\mathbb{E}(D_{n,i}-1)+\cdots+ (\mathbb{E} (D_{n,i}-1))^l \Big)\\
&\le& C (\mathbb{E} (D_{n,i}))^l=C (1+ \log (n-1)-\log (i-1) )^l,~~~~i\ge 2,
\eestar
and similarly,
\bestar
\mathbb{E} (D_{n,1}^l)\le C (1+ \log (n-1))^l.
\eestar
Hence for $n\ge 2$,
\bestar
\sum_{i=1}^n \mathbb{E} (D_{n,i}^l)
&\le& \mathbb{E} (D_{n,1}^l) + \sum_{i=2}^n  2^{l-1} (\mathbb{E} (D_{n,i}-1)^l+1) \\
&\le& C\sum_{i=2}^n   (1+ \log (n-1)-\log (i-1) )^l\\
&\le& C(\log n)^l+
C\int_1^{n-1} (1+\log (n-1)-\log x)^{l}dx\le Cn.
\eestar
The proof of Lemma \ref{lemma1} is complete.
\end{proof}

\subsection{Proof of Lemma \ref{lemma38}}

\begin{proof}[Proof of (\ref{Zln})] 
 In this proof,  let $C_l$ be a constant  depending only on $l$ and $p$ that may take a different value in each appearance.

The initial step is to establish that for any $m \in \mathbb{N}$, ~$\mathbb{E}(Z_{l}(n)) \le C_l b_{l}(n)$ holds for any $n\in \mathbb{N}$ and any $m-1\le l<m$. This will be proved by induction on $m$.

For $m=1$,  the inequality is obvious since $\mathbb{E}(Z_{l}(n))\le \mathbb{E}(\sum_{k=1}^n k \nu_k(n))=n$ for any $0\le l<1$.
Now, assuming the result holds for $m=r\ge 1$, we proceed to prove it for $m=r+1$.
If $\epsilon_{n+1}=0$ and  $U_{n+1}$ belongs to a cluster of size $k$
 in the percolation at time $n$, then we have $\nu_k(n+1)=\nu_k(n)-1, \nu_{k+1}(n+1)=\nu_{k+1}(n)+1$ and $\nu_i(n+1)=\nu_i(n)$ for any $i\not \in \{k, k+1\}$.
Hence, in this case,  we have
$
 Z_{l}(n+1)-Z_{l}(n)= (k+1)^l-k^l.
$
Recalling the definition of $\{\nu_k(n)\}$ in (\ref{nuk}), 
this implies that
\bestar
\mathbb{E}\Big(Z_{l}(n+1)-Z_{l}(n)\Big|\mathscr{H}_n, \epsilon_{n+1}=0\Big)
=\sum_{k=1}^{n} \frac{k\nu_{k}(n)}{n} ((k+1)^l-k^l),
\eestar
where $\mathscr{H}_1=\{\emptyset, \Omega\}$ and $\mathscr{H}_n=\sigma(U_2,\cdots, U_n)$ for $n\ge 2$.
Noting that  $Z_{l}(n+1)-Z_{l}(n)=1$ in the case $\epsilon_{n+1}=1$  gives
\be
\mathbb{E}\Big(Z_{l}(n+1)-Z_{l}(n)\Big|\mathscr{H}_n\Big)=p+(1-p)\sum_{k=1}^{n} \frac{k\nu_{k}(n)}{n} ((k+1)^l-k^l). \label{condrecur}
\ee
Applying Taylor's formula, 
\bestar
(k+1)^l-k^l-lk^{l-1}=\frac{l(l-1)}{2}\theta_k^{l-2}\le l(l-1)2^{|l-2|-1} k^{l-2}
\eestar
holds for $k\ge 1$ and $l\ge 1$, where $\theta_k\in (k, k+1)$.
It follows from (\ref{condrecur})  that
\bestar
\mathbb{E}(Z_{l}(n+1))-\mathbb{E}(Z_{l}(n))&=&p+(1-p)\sum_{k=1}^{n} \frac{k \mathbb{E}(\nu_{k}(n))}{n} ((k+1)^l-k^l)\\
&\le & p+l(1-p)n^{-1}\mathbb{E}(Z_{l}(n))+C_ln^{-1}\mathbb{E}(Z_{l-1}(n)).
\eestar
By  the induction hypothesis, we have
\be
\mathbb{E}(Z_{l}(n+1)) \le  \frac{n+l(1-p)}{n}\mathbb{E}(Z_{l}(n))+C_ln^{-1}b_{l-1}(n). \label{induction}
\ee
If $l(1-p)=1$, then $b_{l-1}(n)=n$ and
$
\mathbb{E}(Z_{l}(n+1))\le (1+1/n)\mathbb{E}(Z_{l}(n))+C_l.
$
Hence,
\bestar
\mathbb{E}(Z_{l}(n))\le C_l \,n\sum_{k=1}^{n} \frac{1}{k}\le C_ln\log n.
\eestar
If $l(1-p)\ne 1$ and $(l-1)(1-p)<1$, then $b_{l-1}(n)=n$ and by  (\ref{induction}),
\bestar
\mathbb{E}(Z_{l}(n+1))+\frac{C_l(n+1)}{l(1-p)-1} \le \frac{n+l(1-p)}{n}\Big(\mathbb{E}(Z_{l}(n))+\frac{C_ln}{l(1-p)-1} \Big).
\eestar
This implies that
\bestar
\mathbb{E}(Z_{l}(n))+\frac{C_ln}{l(1-p)-1} \le \Big(1+\frac{C_l}{l(1-p)-1}\Big)a_{n}(l),
\eestar
where
\be
a_n(l) = \prod_{k=1}^{n-1} \frac{k+l(1-p)}{k} = \frac{\Gamma(n+l(1-p))}{\Gamma(n)\Gamma(l(1-p)+1)}=\frac{n^{l(1-p)}}{\Gamma(l(1-p)+1)} (1+O(n^{-1})),
\label{anl}
\ee
and $\Gamma(\cdot)$ stands for the Gamma function.
Hence
\bestar
\mathbb{E}(Z_{l}(n))\le C_ln^{(l(1-p))\vee 1}.
\eestar
If $(l-1)(1-p)\ge 1$, then by recalling the definition (\ref{bln}), 
we have $b_{l-1}(n+1)-b_{l-1}(n)\le (l-1/2)(1-p)n^{-1}b_{l-1}(n)$ for large $n$. This together with (\ref{induction}) yields that
\bestar
\mathbb{E}(Z_{l}(n+1))+\frac{2C_lb_{l-1}(n+1)}{1-p} \le \frac{n+l(1-p)}{n}\Big(\mathbb{E}(Z_{l}(n))+\frac{2C_lb_{l-1}(n)}{1-p}\Big),
\eestar
and hence
\bestar
\mathbb{E}(Z_{l}(n))\le C_ln^{l(1-p)}.
\eestar
Combining the above facts completes the induction. Therefore,  $\mathbb{E}(Z_{l}(n)) \le C_l b_{l}(n)$ for $l\ge 0$. Similarly, we can obtain that  $\mathbb{E}(Z_{l}(n)) \ge C_l b_{l}(n)$ for $l\ge 0$
and hence the desired result follows.
\end{proof}

\begin{proof}[Proof of (\ref{EZ2na})] 
Let $\mathscr{H}_1=\{\emptyset, \Omega\}$ and $\mathscr{H}_n=\sigma(U_2,\cdots, U_n)$ for $n\ge 2$.
Since $\sum_{k=1}^{n} k\nu_{k}(n)=n$,  it follows from (\ref{condrecur}) that
\be
\mathbb{E}\Big(Z_{2}(n+1)-Z_{2}(n)\Big|\mathscr{H}_n\Big)
=p+(1-p)\sum_{k=1}^{n} \frac{k\nu_{k}(n)}{n} (2k+1)= \frac{2(1-p)}{n} Z_{2}(n)+1.  \label{Z2nind}
\ee
Set $\gamma_n = \frac{n+2(1-p)}{n}$, then we have
\be
\mathbb{E}(Z_{2}(n+1))=\gamma_n \mathbb{E}(Z_{2}(n))+1.  \label{EZ2n}
\ee
If $p=1/2$, then 
\bestar
\mathbb{E}(Z_{2}(n))=n\sum_{k=1}^{n} \frac{1}{k}= n\log n +\gamma n+O(1),
\eestar
where $\gamma$ is Euler's constant.
If $p\ne 1/2$, then
\bestar
\mathbb{E}(Z_{2}(n+1))+\frac{n+1}{1-2p}=\gamma_{n} \Big(\mathbb{E}(Z_{2}(n))+\frac{n}{1-2p}\Big),
\eestar
which implies that
\bestar
\mathbb{E}(Z_2(n))=\frac{2(1-p)a_n(2)-n}{1-2p}=\Big(\frac{n^{2-2p}}{(1-2p)\Gamma(2-2p)}-\frac{n}{1-2p}\Big)(1+O(n^{-1})),
\eestar
where $a_n(2)$ is defined in (\ref{anl}) with $l=2$.
Therefore  $
\mathbb{E}(Z_{2}(n))=b_n (1+O(n^{-1}))$ holds for $p\in (0,1)$.

We will now proceed to estimate $\mbox{Var}(Z_{2}(n))$.
Similar arguments to those in the proof of (\ref{condrecur}) yield that
\bestar
\mathbb{E}\Big((Z_{2}(n+1)-Z_{2}(n))^2\Big|\mathscr{H}_n\Big)
&=&4(1-p)\sum_{k=1}^{n} \frac{(k^3+k^2)\nu_{k}(n)}{n}+1\\
&=&\frac{4(1-p)(Z_{3}(n)+Z_2(n))}{n}+1.
\eestar
This, together with (\ref{Z2nind}),  implies
\bestar
\mathbb{E}\Big(Z_{2}^2(n+1)\Big|\mathscr{H}_n\Big)
&=&\mathbb{E}\Big((Z_{2}(n+1)-Z_{2}(n))^2\Big|\mathscr{H}_n\Big)+2Z_2(n) \mathbb{E}\Big(Z_{2}(n+1)\Big|\mathscr{H}_n\Big) -Z_n^2\\
&=& \frac{4(1-p)}{n}Z_{3}(n)+\gamma_n'Z_2^2(n)+2\gamma_n Z_2(n)+1,
\eestar
where $\gamma_n'=2\gamma_n-1=\frac{n+4(1-p)}{n}$. Hence
\bestar
\mathbb{E}(Z_{2}^2(n+1))= \gamma_n' \mathbb{E}(Z_2^2(n))+\frac{4(1-p)}{n}\mathbb{E}(Z_{3}(n))+2\gamma_n\mathbb{E}(Z_2(n))+1.
\eestar
Applying (\ref{EZ2n}) shows that
\bestar
(\mathbb{E}(Z_{2}(n+1)))^2&=&(\gamma_{n} \mathbb{E}(Z_{2}(n))+1)^2\\
&=&\gamma_n'(\mathbb{E}(Z_2(n)))^2+\frac{4(1-p)^2}{n^2}(\mathbb{E}(Z_2(n)))^2+2\gamma_{n} \mathbb{E}(Z_{2}(n))+1.
\eestar
By letting
\bestar
\alpha_n=\frac{4(1-p)}{n}\mathbb{E}(Z_{3}(n))-\frac{4(1-p)^2}{n^2}(\mathbb{E}(Z_2(n)))^2,
\eestar
we have
\bestar
\mbox{Var}(Z_{2}(n+1)))=\gamma_n' \mbox{Var}(Z_{2}(n)))+\alpha_n,
\eestar
and hence, by noting that  $\mbox{Var}(Z_{2}(1)))=0$,
	\be
	\mbox{Var}(Z_{2}(n+1))) = 
\sum_{j=1}^{n-1} \prod_{k=j+1}^{n} \gamma_k' \alpha_j  +  \alpha_n. \label{svar}
	\ee

Applying (\ref{Zln}) shows that $\alpha_n\le Cn^{-1}b_3(n)\le Cb_4(n)$.
Note that
	\bestar
	 \sum_{j=1}^{n-1} \prod_{k=j+1}^{n} \gamma_k' \alpha_j&\le& C \sum_{j=1}^{n-1} j^{-1}b_3(j) \prod_{k=j+1}^{n} \frac{k+4-4p}{k}\\
     &=& C\sum_{j=1}^{n-1} j^{-1}b_3(j)\exp\left(\sum_{k=j+1}^{n} \log \left(1+\frac{4-4p}{k}\right)\right)\\
	&=& C\sum_{j=1}^{n-1}j^{-1}b_3(j) \exp \left(\sum_{k=j+1}^{n} \left(\frac{4-4p}{k}  +O\left(\frac{1}{k^2}\right) \right) \right) \\
	&\le & C\sum_{j=1}^{n-1} j^{-1}b_3(j)\exp \left((4-4p)\log \frac{n}{j}\right)\\
	&=& C\sum_{j=1}^{n-1}  j^{-1}b_3(j) \left(\frac{j}{n}\right)^{4p-4}\le  Cb_4(n).
	\eestar
It follows from (\ref{svar}) that
	$\mbox{Var}(Z_{2}(n)))\le  Cb_4(n)$.
\end{proof}

\subsection{Proof of Lemma \ref{lemma43}}

\begin{proof}
For any $i<j$ and $0\le x\le (1-p^2)(i\wedge 2)$, we set
\be
a_{i,j}(x)=\prod_{k=i}^{j-1} \Big(1-\frac{x}{k}\Big). \label{aijx}
\ee
Noting that there exists  $c_p>0$ such that $-t\ge \ln(1-t)\ge -t-c_pt^2$ for any $0\le t\le 1-p^2$, we have
\be
a_{i,j}(x)&=&\exp\Big\{\sum_{k=i}^{j-1} \ln \Big(1-\frac{x}{k}\Big) \Big\}=\exp\Big\{-\sum_{k=i}^{j-1} \frac{x}{k}+O(1) i^{-1}\Big\}\nonumber\\
&=& \exp\Big\{-x\ln \Big(\frac{j}{i}\Big)+O(1) i^{-1}\Big\}=\Big(\frac{i}{j}\Big)^{x} (1+O(1)i^{-1}). \label{aijxest}
\ee
Moreover,
\be
&&a_{i,j}(2x)-a_{i,j}^2(x)\nonumber\\
&=& \sum_{l=i}^{j-1} \Big(a_{i,l+1}(2x)a_{l+1,j}^2(x)-a_{i,l}(2x)a_{l,j}^2(x)\Big)=-x^2\sum_{l=i}^{j-1}a_{i,l}(2x)a_{l+1,j}^2(x)l^{-2}\nonumber\\
&=&-x^2 a_{i,j}^2(x)\sum_{l=i}^{j-1} \frac{1}{l^2a_{l,l+1}^2(x)}-x^2 \sum_{l=i}^{j-1}a_{l+1,j}^2(x)l^{-2}(a_{i,l}(2x)-a_{i,l}^2(x))\nonumber\\
&=&-x^2 a_{i,j}^2(x)\sum_{l=i}^{j-1} \frac{1}{l^2(1-x/l)^2}-x^2 \sum_{l=i}^{j-1}a_{l+1,j}^2(x)l^{-2}\sum_{k=i}^{l-1} \Big(a_{i,k+1}(2x)a_{k+1,l}^2(x)-a_{i,k}(2x)a_{k,l}^2(x)\Big)\nonumber\\
&=&-x^2 a_{i,j}^2(x)\sum_{l=i}^{j-1} \frac{1}{l^2(1-x/l)^2}+O(1) a_{i,j}^2(x)\sum_{l=i}^{j-1}\sum_{k=i}^{l-1} \frac{(lk)^{-2}}{(1-x/l)^2(1-x/k)^2}\nonumber\\
&=&-x^2 a_{i,j}^2(x) (i^{-1}-j^{-1}) +O(1)i^{-2}a_{i,j}^2(x), \label{aijdiff}
\ee
where we conventionally define $a_{i,i}(x)=1$ for $x\ge 0$.

As in the proof of Proposition \ref{propa41}, we begin by considering the mean and variance of $\sum_{i=1}^{n-1} p^{D_{n,i}^*}$,
where  $D_{n,i}^*=\sum_{j=i+1}^n I_{i,j}$ for $1\le i< n$.
Then for any $l\in \mathbb{N}$,
\bestar
\mathbb{E} (p^{lD_{n,i}^*})=\prod_{j=i+1}^{n} \mathbb{E}(p^{lI_{i,j}})=a_{i,n}(1-p^l),
\eestar
and for $i<j$,
\bestar
\mathbb{E} (p^{D_{n,i}^*+D_{n,j}^*})= \prod_{k=i+1}^{j} \mathbb{E}(p^{I_{i,k}}) \prod_{k=j+1}^{n} \mathbb{E}(p^{I(U_k\in \{i,j\})})
=a_{i,j}(1-p) a_{j,n}(2-2p).
\eestar
By (\ref{aijxest})  and  (\ref{aijdiff}),  we have
\be
\mathbb{E}\Big(\sum_{i=1}^{n-1} p^{D_{n,i}^*}\Big)=\sum_{i=1}^{n-1}  a_{i,n}(1-p)=\sum_{i=1}^{n-1} \frac{i^{1-p}}{n^{1-p}} +O(1)\sum_{i=1}^{n-1} \frac{i^{-p}}{n^{1-p}}=\frac{n}{2-p}+O(1),  \label{Esump*}
\ee
and
\be
&&\mbox{Var}\Big(\sum_{i=1}^{n-1} p^{D_{n,i}^*}\Big)=\sum_{i=1}^{n-1} \mbox{Var}\Big( p^{D_{n,i}^*}\Big)+2\sum_{1\le i<j\le n-1}\mbox{Cov}\Big(p^{D_{n,i}^*}, p^{D_{n,j}^*}\Big)\nonumber\\
&&~~~~=\sum_{i=1}^{n-1} \Big(a_{i,n}(1-p^2)-a^2_{i,n}(1-p)\Big)+2\sum_{1\le i<j\le n-1}  a_{i,j}(1-p) (a_{j,n}(2-2p)-a_{j,n}^2(1-p))\nonumber\\
&&~~~~=\sum_{i=1}^{n-1} \Big(\frac{i^{1-p^2}}{n^{1-p^2}} - \frac{i^{2-2p}}{n^{2-2p}}\Big) -2(1-p)^2\sum_{1\le i<j\le n-1} \Big(\frac{i}{j}\Big)^{1-p}  \Big(\frac{j}{n}\Big)^{2-2p}   (j^{-1}-n^{-1}) +O(1)\nonumber\\
&&~~~~=\frac{n}{2-p^2}-\frac{n}{3-2p}-2(1-p)^2n\int_0^1\int_x^1 (xy)^{1-p}(y^{-1}-1) dxdy+O(1)\nonumber\\
&&~~~~=\frac{(1-p)^2n}{(2-p^2)(3-2p)}-\frac{(1-p)^2n}{(2-p)^2(3-2p)}+O(1)\nonumber\\
&&~~~~=\frac{2(1-p)^4n}{(2-p^2)(2-p)^2(3-2p)}+O(1)=p^{-2}\sigma_{1}^2n+O(1).    \label{Varsump*}
\ee
Furthermore,
\bestar
\mathbb{E}\Big(\sum_{i=1}^{n} p^{D_{n,i}}\Big)=p\mathbb{E}\Big(\sum_{i=1}^{n-1} p^{D_{n,i}^*}\Big)+(1-p)\mathbb{E}(p^{D_{n,1}^*})+p=\frac{np}{2-p}+O(1),
\eestar
and 
\bestar
&&\mbox{Var}\Big(\sum_{i=1}^{n} p^{D_{n,i}}\Big)=\mbox{Var}\Big(p\sum_{i=1}^{n-1} p^{D_{n,i}^*}+(1-p)p^{D_{n,1}^*}+p\Big)\\
&&~~~~~~~~=p^2\mbox{Var}\Big(\sum_{i=1}^{n-1} p^{D_{n,i}^*}\Big)+(1-p^2)\mbox{Var}(p^{D_{n,1}^*})+2p(1-p)\sum_{i=2}^{n-1}\mbox{Cov}\Big(p^{D_{n,1}^*}, p^{D_{n,i}^*}\Big)\\
&&~~~~~~~~=\sigma_{1}^2 n-2p(1-p)^3\sum_{i=2}^{n-1} \Big(\frac{1}{i}\Big)^{1-p}  \Big(\frac{i}{n}\Big)^{2-2p}   (i^{-1}-n^{-1})+O(1)\\
&&~~~~~~~~=\sigma_{1}^2 n+O(1).
\eestar
The proof of Lemma \ref{lemma43} is complete.
\end{proof}

\subsection{Proof of  Lemma \ref{lemmaadd1}}

\begin{proof}
By using the definitions in (\ref{Iij}) and (\ref{Dni}), when $i<j$ and $I_{i,j}=1$, it follows that
\bestar
D_{n,i}+D_{n,j}-1=I(i\ne 1)+\sum_{k>i, k\ne j} I_{i,k}+I_{i,j}+ \sum_{k> j}I_{j,k}=\sum_{k\ne j}(I_{i,k}+I_{j,k})+1+I(i\ne 1).
\eestar
Consequently,
\be
I_{i,j}p^{D_{n,i}+D_{n,j}-1}=I_{i,j}p^{\sum_{k\ne j}(I_{i,k}+I_{j,k})}p^{1+I(i\ne 1)}.  \label{IpD}
\ee
Moreover, (\ref{IpD}) is also satisfied when $i<j$ and $I_{i,j}=0$.
Therefore, similar arguments to those in the proof of  Lemma \ref{lemma43} show that
\be
\mathbb{E}\Big(\sum_{1\le i<j\le n}I_{i,j}p^{D_{n,i}+D_{n,j}-1}\Big)
&=&\sum_{1\le i<j\le n} p^{1+I(i\ne 1)}\mathbb{E}\Big(I_{i,j}p^{\sum_{k\ne j}(I_{i,k}+I_{j,k})}\Big)\nonumber\\
&=&\sum_{1\le i<j\le n} \frac{p^{1+I(i\ne 1)}}{j-1}\Big(\frac{ij}{n^2}\Big)^{1-p}+O(1)\nonumber\\
&=& \frac{p^2n}{(2-p)(3-2p)}+O(1),  \label{EIji}
\ee
and by (\ref{sigmanT}),
\bestar
\mathbb{E}(\sigma^2(\mathbb{T}_n))
&=&\sum_{i=1}^n (\mathbb{E}(p^{D_{n,i}})-\mathbb{E}(p^{2D_{n,i}}))+2(1-p)\mathbb{E}\Big(\sum_{1\le i<j\le n}I_{i,j}p^{D_{n,i}+D_{n,j}-1}\Big) \\
&=& \frac{np}{2-p}-\frac{np^2}{2-p^2}+\frac{2p^2(1-p)n}{(2-p)(3-2p)}+O(1)=\sigma_{2}^2 n+O(1).
\eestar

To estimate $\mbox{Var}(\sigma^2(\mathbb{T}_n))$, the key step is to estimate the variance of $\sum_{1\le i<j\le n}I_{i,j}p^{D_{n,i}+D_{n,j}-1}$.
As in the proofs of Proposition \ref{propa41} and  Lemma \ref{lemma43}, we begin by studying the variance of $\sum_{1\le i<j\le n}I_{i,j}p^{D_{n,i}^*+D_{n,j}^*-1}$,
where  $D_{n,i}^*=\sum_{j=i+1}^n I_{i,j}$ for $1\le i< n$.

Let $(U_2', \cdots, U_n')$ be an independent copy of $(U_2, \cdots, U_n)$ and define $I'_{i,j}=I(U_j'=i)$ for $1\le i,j\le n$.  
Applying the Efron-Stein inequality (see \cite{ES1981}) yields
\bestar
\mbox{Var}\Big(\sum_{1\le i<j\le n}I_{i,j}p^{D_{n,i}^*+D_{n,j}^*-1}\Big)
&=&\mbox{Var}\Big(\sum_{1\le i<j\le n}I_{i,j}p^{\sum_{k\ne j}(I_{i,k}+I_{j,k})}\Big)\\
&\le& \frac{1}{2}\sum_{l=2}^n \mathbb{E} \Big( \sum_{(i,j)\in \mathcal{I}_l} V_{i,j,l}
+  \sum_{i=1}^{l-1} V_{i, l}\Big)^2\\
&\le& \sum_{l=2}^n \mathbb{E} \Big(  \sum_{(i,j)\in \mathcal{I}_l} V_{i,j,l}\Big)^2
+ \sum_{l=2}^n \mathbb{E} \Big(   \sum_{i=1}^{l-1} V_{i, l}\Big)^2.
\eestar
where $\mathcal{I}_l=\{(i,j): 1\le i<j\le n, i<l, j\ne l\}$,~
$V_{i,l}=p^{\sum_{k\ne l}(I_{j,k}+I_{l,k})}(I_{i,l}-I'_{i,l})$ for $1\le i<l$, and
\bestar
V_{i,j, l}=I_{i,j}\,p^{\sum_{k\ne j, l}(I_{i,k}+I_{j,k})}\big(p^{I_{i,l}+I_{j,l}}-p^{I'_{i,l}+I'_{j,l}}\big),~~~~ (i,j)\in \mathcal{I}_l.
\eestar
If $(i,j), (i',j')\in \mathcal{I}_l$, then
\bestar
 &&\mathbb{E} \Big|\big(p^{I_{i,l}+I_{j,l}}-p^{I'_{i,l}+I'_{j, l}}\big)\big(p^{I_{i', l}+I_{j', l}}-p^{I'_{i', l}+I'_{j', l}}\big)\Big|\\
&\le& \mathbb{P}(I_{i,l}+I_{j,l}+I'_{i,l}+I'_{j, l}+I_{i', l}+I_{j', l}+I'_{i', l}+I'_{j', l}>0)\\
&=& \mathbb{P} (U_l \in \{i,j,i',j'\} ~\mbox{or}~U'_l \in \{i,j,i',j'\})
\le Cl^{-1},
\eestar
and
\bestar
\mathbb{E}(I_{i,j}I_{i',j'})=\frac{1}{(j-1)(j'-1)}I(j\ne j')+\frac{1}{j-1}I(i=i', j=j').
\eestar
Hence, for any $i,j, i', j', l$ with $(i,j), (i',j')\in \mathcal{I}_l$ and $2\le l\le n$,
\bestar
\mathbb{E} |V_{i,j,l}V_{i',j',l}|&\le& \mathbb{E} \Big|\big(p^{I_{i,l}+I_{j,l}}-p^{I'_{i,l}+I'_{j, l}}\big)\big(p^{I_{i', l}+I_{j', l}}-p^{I'_{i', l}+I'_{j', l}}\big)\Big|\mathbb{E}(I_{i,j}I_{i',j'})\\
&\le& \frac{C}{jj'l}I(j\ne j')+\frac{C}{jl}I(i=i', j=j').
\eestar
By noting that
\bestar
\sum_{l=2}^n \sum_{(i,j)\in \mathcal{I}_l} \frac{1}{jl}\le Cn,~~~~\sum_{l=2}^n \sum_{(i,j), (i',j')\in \mathcal{I}_l  \atop |\{i,j,i'j'\}|\le 3} \frac{1}{jj'l}\le Cn,
\eestar
where $|A|$ denotes the number of distinct elements in the set $A$, we have
\be
\sum_{l=2}^n \sum_{(i,j), (i',j')\in \mathcal{I}_l \atop |\{i,j,i'j'\}|\le 3}\mathbb{E} |V_{i,j,l}V_{i',j',l}|\le Cn. \label{vsum}
\ee

If $|\{i,j,i',j'\}|=4$, then 
\bestar
\big(p^{I_{i,l}+I_{j,l}}-p^{I'_{i,l}+I'_{j, l}}\big)\big(p^{I_{i', l}+I_{j', l}}-p^{I'_{i', l}+I'_{j', l}}\big)\ne 0
\eestar
holds if and only if $U_l\in \{i,j\}, U_l'\in \{i',j'\}$ or
$U_l\in \{i',j'\}, U_l'\in \{i,j\}$. Simple calculations show that 
when $|\{i,j,i',j'\}|=4$, we have
\bestar
\big(p^{I_{i,l}+I_{j,l}}-p^{I'_{i,l}+I'_{j, l}}\big)\big(p^{I_{i', l}+I_{j', l}}-p^{I'_{i', l}+I'_{j', l}}\big)\le 0,
\eestar
 and hence
$\mathbb{E} (V_{i,j,l}V_{i',j',l})\le 0$.
This, together with (\ref{vsum}), implies that
\bestar
\sum_{l=2}^n \mathbb{E} \Big( \sum_{(i,j)\in \mathcal{I}_l}V_{i,j,l} \Big)^2\le
\sum_{l=2}^n \sum_{(i,j), (i',j')\in \mathcal{I}_l \atop |\{i,j,i'j'\}|\le 3}\mathbb{E} (V_{i,j,l}V_{i',j',l})\le Cn.
\eestar
Similarly,  we can obtain that
\bestar
\sum_{l=2}^n \mathbb{E} \Big( \sum_{i=1}^{l-1} V_{i,l}\Big)^2\le Cn.
\eestar
Hence
\be
\mbox{Var}\Big(\sum_{1\le i<j\le n}I_{i,j}p^{D_{n,i}^*+D_{n,j}^*-1}\Big) \le Cn. \label{varijl}
\ee
Similarly,
\be
\mbox{Var}\Big(\sum_{i=2}^n I_{1,i}p^{D_{n,1}^*+D_{n,i}^*-1}\Big)\le C.  \label{varijl2}
\ee
It follows from (\ref{varijl}) and (\ref{varijl2}) that
\be
&&\mbox{Var}\Big(\sum_{1\le i<j\le n}I_{i,j}p^{D_{n,i}+D_{n,j}-1} \Big)\nonumber\\
&=&\mbox{Var} \Big(p^2 \sum_{1\le i<j\le n}I_{i,j}p^{D_{n,i}^*+D_{n,j}^*-1}
+(p-p^2)\sum_{i=2}^n I_{1,i}p^{D_{n,1}^*+D_{n,i}^*-1}\Big)\nonumber\\
&\le& 2\mbox{Var} \Big(\sum_{1\le i<j\le n}I_{i,j}p^{D_{n,i}^*+D_{n,j}^*-1}\Big)
+2\mbox{Var} \Big(\sum_{i=2}^n I_{1,i}p^{D_{n,1}^*+D_{n,i}^*-1}\Big)\le  Cn.  \label{varsumiji}
\ee
Note that Lemma \ref{lemma43} yields
\bestar
\mbox{Var} \Big(\sum_{i=1}^{n} p^{D_{n,i}}\Big)\le Cn, ~~~~\mbox{Var} \Big(\sum_{i=1}^{n} p^{2D_{n,i}}\Big)\le Cn.
\eestar
By (\ref{sigmanT}), we have
\bestar
\mbox{Var}(\sigma^2(\mathbb{T}_n))\le 3\mbox{Var} \Big(\sum_{i=1}^{n} p^{D_{n,i}}\Big)+3\mbox{Var} \Big(\sum_{i=1}^{n} p^{2D_{n,i}}\Big)
+12\mbox{Var}\Big(\sum_{1\le i<j\le n}I_{i,j}p^{D_{n,i}+D_{n,j}-1} \Big)\le Cn.
\eestar
The proof of Lemma \ref{lemmaadd1} is complete.
\end{proof}

\subsection{Proof of Lemma \ref{lemmaadd2}}
\begin{proof}
Define $\mu(\mathbb{T}_n)$ and $\sigma^2(\mathbb{T}_n)$ as in
(\ref{munT}) and (\ref{sigmanT}), respectively.
 It follows from  Lemmas \ref{lemma43} and \ref{lemmaadd1} that
\bestar
\mathbb{E}(\nu_1(n))=\mathbb{E}(\mathbb{E}(\nu_1(n)|\mathbb{T}_n))=\mathbb{E}(\mu(\mathbb{T}_n))=\frac{np}{2-p}+O(1)
\eestar
and
\bestar
\mbox{Var}(\nu_1(n))&=&\mbox{Var}(\mathbb{E}(\nu_1(n)|\mathbb{T}_n))+\mathbb{E}(\mbox{Var}(\nu_1(n)|\mathbb{T}_n))\\
&=&\mbox{Var}(\mu(\mathbb{T}_n))+ \mathbb{E}(\sigma^2(\mathbb{T}_n))\le Cn.
\eestar

We will now consider $\nu_2(n)$. Define $I_{i,j}$ and $D_{ni}$ as in (\ref{Iij}) and (\ref{Dni}), respectively. 
Similar arguments to those in the proofs of (\ref{munT}) and (\ref{sigmanT})   show that
\be
\mathbb{E}(\nu_2(n)|\mathbb{T}_n)&=&\sum_{(i,j)\in \mathbb{T}_n, i<j} p^{D_{n,i}+D_{n,j}-2}(1-p)=(1-p)\sum_{1\le i<j\le n}I_{i,j}p^{D_{n,i}+D_{n,j}-2}, ~~~~ \label{Enu2}\\
\mbox{Var}(\nu_2(n)|\mathbb{T}_n)&\le&\sum_{(i,j)\in \mathbb{T}_n, i<j} \Big(p^{D_{n,i}+D_{n,j}-2}(1-p)- p^{2D_{n,i}+2D_{n,j}-4}(1-p)^2\Big)\nonumber\\
&&+2\sum_{(i,j,i',j') \in \mathcal{R}} p^{D_{n,i}+D_{n,j}+D_{n,i'}+D_{n,j'}-5}(1-p)^3\nonumber\\
&\le& \sum_{1\le i<j\le n}I_{i,j}p^{D_{n,i}+D_{n,j}-2}+|\mathcal{R}|,  \label{Varnu2}
\ee
where
\bestar
\mathcal{R}&=&\{(i,j,i',j'): |\{i,j,i',j'\}|=4, (i,j), (i',j')\in \mathbb{T}_n, i<j, i'<j', i<i'\\
&&~~~~~~~~~~~~~~~~\mbox{~and there exists}~ k\in \{i,j\}, k' \in \{i',j'\}~\mbox{such that}~(k,k')\in \mathbb{T}_n \}\\
&=& \{(i,j,i',j')\in \mathcal{I}:  U_j=i, U_{i'}\in \{i,j\}, U_{j'}=i'\},
\eestar
and $\mathcal{I}=\{(i,j,i',j'): ~ |\{i,j,i',j'\}|=4, ~1\le i,j,i',j'\le n, ~i<j, ~i'<j',~i<i'\}$.
Observe that if $ (i,j,i',j')\in \mathcal{I}$, then
\bestar
\mathbb{P}(U_j=i, U_{i'}\in \{i,j\}, U_{j'}=i')\le \frac{2}{(j-1)(i'-1)(j'-1)}.
\eestar
This implies that
\bestar
\mathbb{E}|\mathcal{R}|\le \sum_{i,i',j,j'} \mathbb{P}((i,j,i',j')\in \mathcal{R})\le C\sum_{(i,j,i',j')\in \mathcal{I}} \frac{1}{ji'j'}\le Cn.
\eestar
By applying (\ref{EIji}),  (\ref{varsumiji}), (\ref{Enu2}) and (\ref{Varnu2}), we have
\bestar
\mathbb{E}(\nu_2(n))&=&\mathbb{E}(\mathbb{E}(\nu_2(n)|\mathbb{T}_n))=(1-p)\mathbb{E}\Big(\sum_{1\le i<j\le n} I_{i,j}p^{D_{n,i}+D_{n,j}-2} \Big)\\
&=&\frac{p(1-p)n}{(2-p)(3-2p)}+O(1)
\eestar
and
\bestar
\mbox{Var}(\nu_2(n))&=&\mbox{Var}(\mathbb{E}(\nu_2(n)|\mathbb{T}_n))+\mathbb{E}(\mbox{Var}(\nu_2(n)|\mathbb{T}_n))\\
&\le &\mbox{Var}\Big(\sum_{1\le i<j\le n}I_{i,j}p^{D_{n,i}+D_{n,j}-2}\Big)+ \mathbb{E}\Big(\sum_{1\le i<j\le n}I_{i,j}p^{D_{n,i}+D_{n,j}-2}\Big)+\mathbb{E}|\mathcal{R}|\\
&\le& Cn.
\eestar
The proof of Lemma \ref{lemmaadd2} is complete.
\end{proof}

\renewcommand{\thesection}{Appendix B}
\section{Proof of (\ref{BoundDij})}  \label{sectappendix2}

Before the proof of (\ref{BoundDij}), we introduce some notation and two lemmas.
For any $k_1,k_2\in \mathbb{N}$, the Kronecker delta function $\delta_{k_1, k_2}$
is defined by $\delta_{k_1, k_2}=1$ if $k_1=k_2$, and $\delta_{k_1, k_2}=0$ otherwise. For any finite set $A$, we denote by $|A|$ the number of distinct elements
in $A$.

\begin{lemma} \label{lemmacount}
For any $k_1,k_2,k_1',k_2'\in \mathbb{N}$, we have
\be
&&\delta_{k_1, k_2}+\delta_{k_2, k_2'}+|\{k_1,k_2,k_2'\}|\in \{2, 3\}, \label{3k} 
\\
&&\delta_{k_1, k_2}+\delta_{k_1', k_2'}+\delta_{k_2, k_2'}+|\{k_1,k_2,k_1',k_2'\}|\in \{2, 3, 4\}, \label{4k}
\ee 
and  the expression in \eqref{4k} equals $2$ if and only if $k_1=k_2'\ne k_2=k_1'$.
\end{lemma}

\begin{proof}
Fix $k_1,k_2,k_1',k_2'\in \mathbb{N}$ and write $$a_0:=\delta_{k_1, k_2}+\delta_{k_1', k_2'}+\delta_{k_2, k_2'}+|\{k_1,k_2,k_1',k_2'\}|.$$
We only prove the result  for  $a_0$. The proof of \eqref{3k} is similar to that of \eqref{4k} but simpler, and is therefore omitted.   

We proceed by considering the possible values of  $|\{k_1,k_2,k_1',k_2'\}|$ separately.  

(1)  $|\{k_1,k_2,k_1',k_2'\}|=1$. Then $k_1=k_2=k_1'=k_2'$, and hence
$a_0=4$.

(2)   $|\{k_1,k_2,k_1',k_2'\}|=2$.  Then  $\delta_{k_1, k_2}+\delta_{k_1', k_2'}+\delta_{k_2, k_2'}\in \{0, 1, 2\}$.

Indeed,  if  the sum  equals $3$, then   $\delta_{k_1, k_2}=\delta_{k_1', k_2'}=\delta_{k_2, k_2'}=1$,  which implies $k_1=k_2=k_1'=k_2'$. This gives  $|\{k_1,k_2,k_1',k_2'\}|=1$,  contradicting the assumption that $|\{k_1,k_2,k_1',k_2'\}|=2$. 

 In this case,  $a_0\in \{2,3,4\}$.

(3)   $|\{k_1,k_2,k_1',k_2'\}|=3$.  Then   $\delta_{k_1, k_2}+\delta_{k_1', k_2'}+\delta_{k_2, k_2'}\in \{0, 1\}$. 

In fact, if  the sum  equals $3$, then $k_1=k_2=k_1'=k_2'$,   contradicting the assumption that  $|\{k_1,k_2,k_1',k_2'\}|=3$. If  the sum  equals $2$, then one of the following must hold:  $k_1=k_2=k_2'$, or $k_1=k_2$ and $k_1'=k_2'$, or $k_1=k_1'=k_2'$.
In each of these subcases, $|\{k_1,k_2,k_1',k_2'\}|\le 2$, also contradicting the assumption that $|\{k_1,k_2,k_1',k_2'\}|=3$.  

Hence in this case, $a_0\in \{3,4\}$.

(4)   $|\{k_1,k_2,k_1',k_2'\}|=4$.  Then $\delta_{k_1, k_2}=\delta_{k_1', k_2'}=\delta_{k_2, k_2'}=0$, and hence $a_0=4$.

From the above cases,  \eqref{4k}  is obtained.
Moreover,   $a_0=2$ if and only if  $|\{k_1,k_2,k_1',k_2'\}|=2$ and $\delta_{k_1, k_2}=\delta_{k_1', k_2'}=\delta_{k_2, k_2'}=0$, which is  equivalent to $k_1=k_2'\ne k_2=k_1'$.  
\end{proof}

\begin{lemma} \label{lemmacount2} Assume that $i,i',j$ are distinct positive integers. Define
\be 
L_{i,i',j}=\sum_{k_1=1}^{i-1}\sum_{k_2=1}^{i\wedge j-1}\sum_{k_1'=1}^{i'-1}\sum_{k_2'=1}^{i'\wedge j-1}
i^{-2+\delta_{k_1, k_2}}(i')^{-2+\delta_{k_1', k_2'}}j^{-2+\delta_{k_2, k_2'}}.
\label{defLiij}
\ee 
Then we have
\be 
L_{i,i',j}\le C\sum_{(r_1, r_2, r_3)\in \mathcal{R}_1\cup \mathcal{R}_2} l_{1}^{r_1}  l_{2}^{r_2} l_{3}^{r_3},
\label{Liij0}
\ee 
where  $l_1< l_2<l_3$ are the values of $i,i',j$ arranged in increasing order,  and
\bestar
&&\mathcal{R}_1=\big\{(a,b,c): ~a\in \{0,1,2\},~a+b=-2, ~ c=\{-2,-1\}\big\},\\
&&\mathcal{R}_2=\big\{(a, b, c): ~a\in \{0,1,2\}, ~a+b\in\{-1,0\}, ~a+b+c\in\{-4, -3,-2\} \big\}.
\eestar
\end{lemma}

\begin{proof}
Note that $L_{i,i',j}=L_{i',i,j}$. Without  loss of generality, we may assume 
$i<i'$.
For any $\mathbf{k}:=(k_1,k_2,k_1',k_2')\in \mathbb{N}^4$, we  define 
\bestar
h(\mathbf{k})=(\delta_{k_1, k_2}, \delta_{k_1', k_2'}, \delta_{k_2, k_2'}, N_i(\mathbf{k}), N_{i'}(\mathbf{k}), N_j(\mathbf{k})),
\eestar
where $N_l(\mathbf{k})$ denotes the number of distinct elements in $\{k_1,k_2,k_1',k_2'\}$ that are less than $l$, for $l\in \{i,i',j\}$.  Denote 
\bestar
\mathcal{R}_h:=\{h(\mathbf{k}):  1\le k_1,k_2<i, ~1\le k_1', k_2' <i', ~ k_2,k_2'<j\},
\eestar
and for any $\mathbf{a}=(a_1,\cdots, a_6)\in \mathcal{R}_h$, set
\bestar
h^{-1}(\{\mathbf{a}\}):=\{\mathbf{k}:  
h(\mathbf{k})=\mathbf{a},  1\le k_1,k_2<i, ~1\le k_1', k_2' <i', ~ k_2,k_2'<j\}.
\eestar
Then $|\mathcal{R}_h|\le 2^3\cdot 4^3=512$ and
\bestar
L_{i,i',j}&=&\sum_{\mathbf{a}\in \mathcal{R}_h}\sum_{\mathbf{k}\in h^{-1}(\{\mathbf{a}\})}
i^{-2+\delta_{k_1, k_2}}(i')^{-2+\delta_{k_1', k_2'}}j^{-2+\delta_{k_2, k_2'}}\\
&=&\sum_{\mathbf{a}\in \mathcal{R}_h} i^{-2+a_1}(i')^{-2+a_2}j^{-2+a_3} |h^{-1}(\{\mathbf{a}\})|.
\eestar

We now prove the desired result according to the relative order of $i,i',j$.

\vskip 0.3cm
\noindent {\bf Case 1}:  $i<i'<j$.

  For any $\mathbf{a}=(a_1,\cdots, a_6)\in \mathcal{R}_h$, 
we have $|h^{-1}(\{\mathbf{a}\})|<i^{a_4}(i')^{a_5-a_4}j^{a_6-a_5}$,
and hence
\be 
L_{i,i',j}
&=&\sum_{\mathbf{a}\in \mathcal{R}_h} i^{-2+a_1+a_4}(i')^{-2+a_2+a_5-a_4}j^{-2+a_3+a_6-a_5}.  \label{Liij}
\ee 
Fix  $\mathbf{a}=(a_1,\cdots, a_6)\in \mathcal{R}_h$ and $\mathbf{k}=(k_1,k_2,k_1',k_2')\in h^{-1}(\{\mathbf{a}\})$, 
and define  
\bestar
r_1=-2+a_1+a_4, ~~r_2=-2+a_2+a_5-a_4, ~~r_3=-2+a_3+a_6-a_5.
\eestar
Then  $(r_1,r_2,r_3)$ satisfies the following properties:
\begin{enumerate}
\item[$(a)$]  $r_1\in \{0, 1,2\}$; 
\item[$(b)$]  $r_1+r_2\in \{-2,-1, 0\}$; 
\item[$(c)$] $r_1+r_2+r_3\in\{-4, -3, -2\}$;
\item[$(d)$]   If $r_1+r_2=-2$, then $r_1+r_2+r_3\in\{-4, -3\}$.
\end{enumerate}
Therefore,  combining \eqref{Liij} with $(a)-(d)$ yields \eqref{Liij0}.

We now verify $(a)-(d)$.
The result $(a)$ holds since
$r_1 \ge -2+\delta_{k_1,k_2}+|\{k_1,k_2\}|=0$
and  
\bestar
r_1\le -2+\delta_{k_1,k_2}+|\{k_1,k_2,k_1',k_2'\}|\le  -2+\delta_{k_1,k_2}+|\{k_1,k_2\}|+2=2.
\eestar
 It follows from Lemma \ref{lemmacount} that 
\bestar
r_1+r_2+r_3=\delta_{k_1, k_2}+\delta_{k_1', k_2'}+\delta_{k_2, k_2'}+|\{k_1,k_2,k_1',k_2'\}|-6\in\{-4, -3, -2\},
\eestar
which establishes $(c)$.
By  using  Lemma \ref{lemmacount} and noting that $r_1+r_2=-4+a_1+a_2+a_5$, we  obtain
\bestar
r_1+r_2=\delta_{k_1, k_2}+\delta_{k_1', k_2'}+\delta_{k_2, k_2'}+|\{k_1,k_2,k_1',k_2'\}|-4-\delta_{k_2, k_2'}\in
\{-2,-1,0\}.
\eestar
Hence (b) holds.   Moreover,  if $r_1+r_2=-2$, then $\delta_{k_1, k_2}+\delta_{k_1', k_2'}+\delta_{k_2, k_2'}+|\{k_1,k_2,k_1',k_2'\}|\in\{2,3\}$,
 and consequently $r_1+r_2+r_3\in\{-4,-3\}.$  This proves $(d)$.

\vskip 0.3cm
\noindent
{\bf Case 2}: $i<j<i'$.  

 For any $\mathbf{a}=(a_1,\cdots, a_6)\in \mathcal{R}_h$, 
we have $|h^{-1}(\{\mathbf{a}\})|<i^{a_4}j^{a_6-a_4}(i')^{a_5-a_6}$,
and consequently,
\bestar
L_{i,i',j}
&=&\sum_{\mathbf{a}\in \mathcal{R}_h} i^{-2+a_1+a_4}j^{-2+a_3+a_6-a_4}(i')^{-2+a_2+a_5-a_6}.
\eestar
Fix  $\mathbf{a}=(a_1,\cdots, a_6)\in \mathcal{R}_h$ and $\mathbf{k}=(k_1,k_2,k_1',k_2')\in h^{-1}(\{\mathbf{a}\})$, 
and define  
\bestar
r_1'=-2+a_1+a_4, ~~r_2'=-2+a_3+a_6-a_4, ~~r_3'=-2+a_2+a_5-a_6.
\eestar
Then $(r_1',r_2',r_3')$ satisfies $(a)$ and $(c)$ by  the same arguments as in {\bf Case 1}.
Since $r_1'+r_2'=-4+a_1+a_3+a_6$, it follows from Lemma \ref{lemmacount} that
\bestar
&& r_1'+r_2'\le -4+\delta_{k_1, k_2}+\delta_{k_2, k_2'}+|\{k_1,k_2,k_1',k_2'\}|\le 0,
\\
&& r_1'+r_2'\ge  -4+\delta_{k_1, k_2}+\delta_{k_2, k_2'}+|\{k_1,k_2,k_2'\}|\ge -2.
\eestar
Hence $(b)$ holds for $(r_1',r_2',r_3')$.
Next we prove that $(d)$ also holds for $(r_1',r_2',r_3')$.
If $r_1'+r_2'=-2$ and $k_2'< j$, then 
$
\delta_{k_1, k_2}+\delta_{k_2, k_2'}+|\{k_1,k_2,k_1',k_2'\}|=r_1'+r_2'+4=2,
$
which implies $$r_1'+r_2'+r_3'=\delta_{k_1, k_2}+\delta_{k_1', k_2'}+\delta_{k_2, k_2'}+|\{k_1,k_2,k_1',k_2'\}|-6\in \{-4, -3\}.$$
If $r_1'+r_2'=-2$ and $i'>k_2'\ge j$, then 
$
 \delta_{k_1, k_2}+\delta_{k_2, k_2'}+|\{k_1,k_2,k_2'\}|=r_1'+r_2'+4=2,
$
and consequently,
\bestar
r_1'+r_2'+r_3'&=&\delta_{k_1, k_2}+\delta_{k_1', k_2'}+\delta_{k_2, k_2'}+|\{k_1,k_2,k_1',k_2'\}|-6\\
&=&\delta_{k_1, k_2}+\delta_{k_2, k_2'}+|\{k_1,k_2,k_2'\}|-5=-3.
\eestar

Hence  $(r_1',r_2',r_3')$ satisfies $(a)-(d)$.
This, together with \eqref{Liij},  yields  \eqref{Liij0}.

\vskip 0.3cm
\noindent
{\bf Case 3}: $j<i<i'$. 

 For any $\mathbf{a}=(a_1,\cdots, a_6)\in \mathcal{R}_h$, 
we have $|h^{-1}(\{\mathbf{a}\})|<j^{a_6}i^{a_4-a_6}(i')^{a_5-a_4}$,
and consequently,
\bestar
L_{i,i',j}
&=&\sum_{\mathbf{a}\in \mathcal{R}_h} j^{-2+a_3+a_6} i^{-2+a_1+a_4-a_6}(i')^{-2+a_2+a_5-a_4}.
\eestar
Fix  $\mathbf{a}=(a_1,\cdots, a_6)\in \mathcal{R}_h$ and $\mathbf{k}=(k_1,k_2,k_1',k_2')\in h^{-1}(\{\mathbf{a}\})$, 
and define 
\bestar
r_1''=-2+a_3+a_6, ~~r_2''=-2+a_1+a_4-a_6, ~~r_3''=-2+a_2+a_5-a_4.
\eestar
By arguments similar to those in  {\bf Case 1} and {\bf Case 2}, 
we obtain that $(r_1'',r_2'',r_3'')$ satisfies $(a)-(d)$.
Therefore,  \eqref{Liij0} follows from  \eqref{Liij}.

The proof of Lemma \ref{lemmacount2}  is complete.
\end{proof}

\begin{proof}[Proof of (\ref{BoundDij})]
It follows from (\ref{Diij}) that
\bestar
\sum_{j=2}^n \mathbb{E}\Big(\sum_{i\ne j} |\Delta_i||\Delta_{ij}|\Big)^2&\le & \sum_{\mathbf{I}\in \mathcal{P}} \mathbb{E}(Y_{k_1, i}Y_{k_2, i}Y_{k_2,j}Y_{k_1',i'}Y_{k_2', i'}Y_{k_2', j})\\
&=&\sum_{\mathbf{I}\in \mathcal{P}_1}  \mathbb{E}(Y_{k_1, i}Y_{k_2, i}) \mathbb{E}(Y_{k_1',i'}Y_{k_2', i'})\mathbb{E}(Y_{k_2, j}Y_{k_2',j})\\
&&~~~~+\sum_{\mathbf{I}\in \mathcal{P}_2}  \mathbb{E}(Y_{k_1, i}Y_{k_2, i} Y_{k_1',i}Y_{k_2', i})\mathbb{E}(Y_{k_2, j}Y_{k_2',j})\\
&:=&I_{n,1}+I_{n,2},
\eestar
where $\mathbf{I}:=(i, j, k_1,k_2, i', k_1', k_2'),~\mathcal{P}_1=\mathcal{P}\cap \{\mathbf{I}: i\ne i'\}, ~\mathcal{P}_2= \mathcal{P}\cap \{\mathbf{I}: i=i'\}$ and
  $$ \mathcal{P}:=\{\mathbf{I}: 1\le k_1,k_2<i\le n, ~1\le k_1', k_2' <i'\le n, ~ k_2,k_2'<j\le n,~i,i'\ne j \}.$$

For any  integers $k_1,k_2, i$ satisfying $1\le k_1,k_2<i\le n$, we have
\be
 \mathbb{E}(Y_{k_1, i}Y_{k_2, i}) &=&\frac{2}{(i-1)^{2}}I(k_1\ne k_2)+\frac{2(i-2)}{(i-1)^{2}}I(k_1=k_2)\nonumber\\
&\le & C\Big(i^{-2}I(k_1\ne k_2)+i^{-1}I(k_1=k_2)\Big)=Ci^{-2+\delta_{k_1, k_2}}.  \label{Ykii}
\ee
This implies that, for any $\mathbf{I}\in \mathcal{P}_1$,
\bestar
\mathbb{E}(Y_{k_1, i}Y_{k_2, i}) \mathbb{E}(Y_{k_1',i'}Y_{k_2', i'})\mathbb{E}(Y_{k_2, j}Y_{k_2',j})
\le Ci^{-2+\delta_{k_1, k_2}}(i')^{-2+\delta_{k_1', k_2'}}j^{-2+\delta_{k_2, k_2'}}.
\eestar

We first prove that $I_{n,1}\le Cn$. It  suffices to show that
\bestar
L_n:=\sum_{\mathbf{I}\in \mathcal{P}_1} i^{-2+\delta_{k_1, k_2}}(i')^{-2+\delta_{k_1', k_2'}}j^{-2+\delta_{k_2, k_2'}}
=\sum_{\substack{1\le i,i',j\le n\\
i,i', j~\text{distinct}}} L_{i,i',j}\le Cn,
\eestar
where $L_{i,i',j}$ is defined in \eqref{defLiij}.
Since $L_{i,i',j}=L_{i',i,j}$, we  obtain
\bestar
\frac{L_n}{2}=\sum_{1\le i<i'<j\le n} L_{i,i',j} +\sum_{1\le i<j<i'\le n} L_{i,i',j} +\sum_{1\le j<i<i'\le n} L_{i,i',j} :=L_{1n}+L_{2n}+L_{3n}.
\eestar
It  follows from   Lemma \ref{lemmacount2}  that
\bestar
L_{1n}&\le&C\sum_{(r_1, r_2, r_3)\in \mathcal{R}_1\cup \mathcal{R}_2} \sum_{1\le i<i'<j\le n} i^{r_1}(i')^{r_2} j^{r_3},
\eestar
where $\mathcal{R}_1$ and  $\mathcal{R}_2$ are defined in  Lemma \ref{lemmacount2}.
For any $(r_1, r_2, r_3)\in \mathcal{R}_1$,  since  $r_1\ge 0 $ and $ r_1+r_2=-2$, we have
\bestar
\sum_{1\le i<i'<j\le n} i^{r_1}(i')^{r_2} j^{r_3}
&\le& \sum_{j=3}^n \sum_{i'=2}^j \sum_{i=1}^{i'}  i^{r_1}(i')^{r_2} j^{r_3}
\le C\sum_{j=3}^n \sum_{i'=2}^j  (i')^{r_1+r_2+1}j^{r_3}\\
&\le& C\sum_{j=3}^n  j^{r_3}\log j
\le \left\{
\begin{array}{ll}
C, & r_3=-2;\\
C \log^2 n, & r_3=-1.
\end{array}
\right.
\eestar
For any $(r_1, r_2, r_3)\in \mathcal{R}_2$,  since  $r_1\ge 0 $ and $ r_1+r_2+1\ge 0$, we have
\bestar
\sum_{1\le i<i'<j\le n} i^{r_1}(i')^{r_2} j^{r_3}
&\le& \sum_{j=1}^n \sum_{i'=1}^j \sum_{i=1}^{i'}  i^{r_1}(i')^{r_2} j^{r_3}
\le C\sum_{j=1}^n \sum_{i'=1}^j  (i')^{r_1+r_2+1}j^{r_3}\\
&\le& C\sum_{j=1}^n j^{r_1+r_2+r_3+2}
\le \left\{
\begin{array}{ll}
C, & r_1+r_2+r_3=-4;\\
C\log n, & r_1+r_2+r_3=-3;\\
Cn, & r_1+r_2+r_3=-2.\\
\end{array}
\right.
\eestar
Hence $L_{1n}\le Cn$. Similarly, we obtain $L_{2n}\le Cn$ and $L_{3n}\le Cn$. Therefore,
$L_n\le Cn$ and consequently
\be
I_{n,1}\le Cn.  \label{In1}
\ee

Note that for $k_1,k_2,k_1',k_2'<i$, we have
\be
\mathbb{E}(Y_{k_1, i}Y_{k_2, i}Y_{k_1',i}Y_{k_2', i})&=&
\left\{
\begin{array}{cc}
0, &  |\{k_1, k_2, k_1', k_2'\}|\ge 3\\
2(i-1)^{-2}, &  |\{k_1, k_2, k_1', k_2'\}|=2\\
2(i-2)(i-1)^{-2}, &  |\{k_1, k_2, k_1', k_2'\}|=1
\end{array}
\right.\nonumber\\
&\le & Ci^{-|\{k_1, k_2, k_1', k_2'\}|}.  \label{EYi4}
\ee
It follows from (\ref{Ykii}) and  (\ref{EYi4}) that
\bestar
\mathbb{E}(Y_{k_1, i}Y_{k_2, i} Y_{k_1',i}Y_{k_2', i})\mathbb{E}(Y_{k_2, j}Y_{k_2',j})\le Ci^{-|\{k_1, k_2, k_1', k_2'\}|} j^{-1}.
\eestar
Similar arguments to those in the proof of (\ref{In1}) show that $I_{n, 2}\le Cn$ and hence
\be
\sum_{j=2}^n \mathbb{E}\Big(\sum_{i\ne j} |\Delta_i||\Delta_{ij}|\Big)^2\le I_{n, 1}+I_{n, 2} \le Cn.    \label{Deltai21B}
\ee

Similarly, by (\ref{Deltaia}) and (\ref{EYi4}), we also have
\bestar
\sum_{i=2}^n \mathbb{E} (\Delta_i^4 ) &\le&  \sum_{i=2}^n \mathbb{E} \Big(\sum_{k=1}^{i-1} Y_{k,i}\Big)^4=\sum_{i=2}^n \sum_{k_1,k_2,k_1',k_2'=1}^{i-1} \mathbb{E}(Y_{k_1, i}Y_{k_2, i}Y_{k_1',i}Y_{k_2', i})\nonumber\\
&\le& C\sum_{i=2}^n \sum_{k_1,k_2,k_1',k_2'=1}^{i-1} i^{-|\{k_1, k_2, k_1', k_2'\}|}\le Cn.    
\eestar
This, together with (\ref{Deltai21B}), implies (\ref{BoundDij}).
\end{proof}

\end{appendices}

\end{document}